\documentclass[a4paper,11pt]{article}

\usepackage[margin=1.4in]{geometry}

\usepackage{amsmath}
\usepackage{algorithmic}
\usepackage{algorithm}
\usepackage{amsthm}
\usepackage{amsfonts}
\usepackage{comment}
\usepackage{graphicx}
\usepackage{hyperref}
\usepackage{color}
\usepackage{multirow}

\newtheorem{theorem} {Theorem}
\newtheorem{lemma} {Lemma}

\newtheorem{corollary} {Corollary}
\newtheorem{assumption} {Assumption}
\newtheorem{observation} {Observation}
\newtheorem{proposition} {Proposition}

\def\u{{\mathbf{u}}}
\def\v{{\mathbf{v}}}

\def\X{{\mathbf{X}}}
\def\Y{{\mathbf{Y}}}
\def\A{{\mathbf{A}}}

\def\B{{\mathbf{B}}}

\def\G{{\mathbf{G}}}
\def\V{{\mathbf{V}}}
\def\Z{{\mathbf{Z}}}

\def\U{{\mathbf{U}}}

\def\R{{\mathbf{R}}}

\def\matE{{\mathbf{E}}}

\newcommand{\mT}{\mathcal{T}}

\newcommand{\mX}{\mathcal{X}}

\newcommand{\mS}{\mathcal{S}}
\newcommand{\mbS}{\mathbb{S}}

\newcommand{\E}{\mathbb{E}}

\newcommand{\projgrad}{\mathcal{G}}

\newcommand{\trace}{\textrm{Tr}}
\newcommand{\rank}{\textrm{rank}}
\newcommand{\reals}{\mathbb{R}}
\newcommand{\sym}{\mathbb{S}}

\newcommand{\ball}{\mathbb{B}}

\title{On the Convergence of Projected-Gradient Methods with Low-Rank Projections for Smooth Convex Minimization over Trace-Norm Balls and Related Problems}
\date{}
\author{Dan Garber \\
Technion - Israel Institute of Technology \\
{\small{dangar@technion.ac.il}}}

\begin{document}
 \maketitle

\begin{abstract}
Smooth convex minimization over the unit trace-norm ball is an important optimization problem in machine learning, signal processing, statistics and other fields, that underlies many tasks in which one wishes to recover a low-rank matrix given certain measurements. While first-order methods for convex optimization enjoy optimal convergence rates, they require in worst-case to compute a full-rank singular value decomposition on each iteration in order to compute the Euclidean projection onto the trace-norm ball. These full-rank SVD computations however prohibit the application of such methods to large-scale problems. A simple and natural heuristic to reduce the computational cost of such methods is to approximate the Euclidean projection using only a low-rank singular value decomposition. This raises the question if, and under what conditions, this simple heuristic can indeed result in provable convergence to optimal solutions.

In this paper we show that any optimal solution is a center of a Euclidean ball inside-which the projected-gradient mapping admits rank that is at most the multiplicity of the largest singular value of the gradient vector at this optimal point. Moreover, the radius of the ball scales with the spectral gap of this gradient vector. We show how this readily implies the local convergence (i.e., from a "warm-start" initialization) of standard first-order methods such as the projected-gradient method and accelerated gradient methods, using only low-rank SVD computations. We also quantify the effect of "over-parameterization", i.e., using SVD computations with higher rank, on the radius of this ball, showing it can increase dramatically with moderately larger rank. We extend our results also to the setting of smooth convex minimization with trace-norm regularization and smooth convex optimization over bounded-trace positive semidefinite matrices. Our theoretical investigation is supported by concrete empirical evidence that demonstrates the \textit{correct} convergence of first-order methods with low-rank projections for the matrix completion task on real-world datasets.
\end{abstract}

\section{Introduction} 
The main subject of investigation in this paper is the following optimization problem:
\begin{eqnarray}\label{eq:optProb}
\min_{\Vert{\X}\Vert_*\leq 1}f(\X),
\end{eqnarray}
where $f:\reals^{m\times n}\rightarrow\reals$ is convex and $\beta$-smooth (i.e., gradient-Lipschitz), and $\Vert{\cdot}\Vert_*$ denotes the trace-norm, i.e., sum of singular values (aka the nuclear norm).

Problem \eqref{eq:optProb} has received much attention in recent years and has many applications in machine learning, signal processing, statistics and engineering, such as the celebrated matrix completion problem \cite{Candes09, Recht11, Jaggi10}, affine rank minimization problems \cite{recht2010guaranteed, Prateek10}, robust PCA \cite{Candes11}, and more.

Many standard first-order methods such as the projected-gradient descent \cite{Nesterov13}, Nesterov's accelerated gradient method \cite{Nesterov13} and FISTA \cite{FISTA}, when applied to Problem \eqref{eq:optProb}, require on each iteration to compute the projected-gradient mapping w.r.t. the trace-norm ball, given by
\begin{eqnarray}\label{eq:traceNormProj}
\Pi_{\Vert\cdot\Vert_*\leq 1}[\X-\eta\nabla{}f(\X)],
\end{eqnarray}
for some current iterate $\X\in\reals^{m\times n}$ and step-size $\eta>0$, where $\Pi_{\Vert\cdot\Vert_*\leq 1}[\cdot]$ denotes the Euclidean projection onto the unit trace-norm ball. 

It is well known that computing the projection step in \eqref{eq:traceNormProj} amounts to computing the singular value decomposition of the matrix $\Y = \X-\eta\nabla{}f(\X)$ and projecting the vector of singular values onto the unit simplex (keeping the left and right singular vectors without change). Unfortunately, in worst-case, a full-rank SVD computation is required which amounts to $O(mn^2)$ runtime per iteration, assuming $m\leq n$. This naturally prohibits the use of such methods for large scale problems in which both $m,n$ are large.

Since \eqref{eq:traceNormProj} requires in general expensive full-rank SVD computations, a very natural and simple heuristic to reduce the computational complexity is to replace the expensive projection operation $\Pi_{\Vert\cdot\Vert_*\leq 1}[\Y]$ with an  approximate "projection", which (accurately) projects the matrix $\widehat{\Y}_r$ - the best rank-$r$ approximation of $\Y$. That is, we consider replacing $\Pi_{\Vert\cdot\Vert_*\leq 1}[\Y]$ with the operation
\begin{eqnarray*}
\widehat{\Pi}^r_{\Vert\cdot\Vert_*\leq 1}[\Y] := \Pi_{\Vert{\cdot}\Vert_*\leq 1}[\widehat{\Y}_r] = \Pi_{\Vert\cdot\Vert_*\leq 1}[\U_r\Sigma_r\V_r^{\top}],
\end{eqnarray*}
where $\U_r\Sigma_r\V_r^{\top}$ corresponds to the rank-$r$ truncated SVD of $\Y$ (i.e., we consider only the top $r$ components of the SVD).

Using state-of-the-art Krylov subspace methods, such as the Power Iterations algorithm or the Lanczos algorithm (see for instance the classical text \cite{golub2012matrix} and also the recent works \cite{musco2015randomized, allen2016lazysvd}), $\widehat{\Pi}^r_{\Vert\cdot\Vert_*\leq 1}[\cdot]$ could be computed with runtime proportional to only $O(rmn)$ - a very significant speedup when $r << \min\{m,n\}$. Moreover, in many problems the gradient matrix is sparse (e.g., in the well studied matrix completion problem), in which case further significant accelerations apply. The drawback of course is that when using the approximated procedure $\widehat{\Pi}^r_{\Vert\cdot\Vert_*\leq 1}[\Y]$ , the highly desired convergence guarantees of first-order methods need no longer hold. 

A motivation for the plausible effectiveness of this heuristic is that Problem \eqref{eq:optProb} is often used as a convex relaxation for non-convex rank-constrained optimization problems, which are often assumed to admit a low-rank global minimizer (as is in all examples provided above). Given this low-rank structure of the optimal solution, one may wonder if indeed storing and manipulating high-rank matrices when optimizing \eqref{eq:optProb} is mandatory, or alternatively, that at some stage during the run of the algorithm the iterates all become low-rank.

It is thus natural to ask: \textit{under which conditions is it possible to replace the projection $\Pi_{\Vert{\cdot}\Vert_*\leq 1}[\cdot]$ with the approximation $\widehat{\Pi}^r_{\Vert\cdot\Vert_*\leq 1}[\cdot]$, while keeping the original convergence guarantees of first-order methods?} 

Or, put differently, we ask \textit{for which $\X,r$ (and a suitable $\eta$) does 
\begin{eqnarray}\label{eq:equalProj}
\Pi_{\Vert\cdot\Vert_*\leq 1}[\X-\eta\nabla{}f(\X)] = \widehat{\Pi}^r_{\Vert\cdot\Vert_*\leq 1}[\X-\eta\nabla{}f(\X)]
\end{eqnarray}
hold?}

Our main result in this paper is the formulation and proof of the following proposition, presented at this point only informally.
\begin{proposition}\label{prop:main}
For any optimal solution $\X^*$ to Problem \eqref{eq:optProb}, if the truncated-SVD rank parameter $r$ satisfies $r \geq r_0=\#\sigma_1(\nabla{}f(\X^*))$ - the multiplicity of the largest singular value in the gradient vector $\nabla{}f(\X^*)$ then, there exists a Euclidean ball centered at $\X^*$, inside-which \eqref{eq:equalProj} holds. Moreover, the radius of the ball scales with the spectral gap $\sigma_1(\nabla{}f(\X^*)) - \sigma_{r_0+1}(\nabla{}f(\X^*))$.
\end{proposition}

%In this paper we show that for any optimal solution $\X^*$ to Problem \eqref{eq:optProb} and as long as the truncated-SVD rank parameter $r$ satisfies $r \geq \#\sigma_1(\nabla{}f(\X^*))$ - the multiplicity of the largest singular value of the gradient vector evaluated at $\X^*$, there exists a non-empty Euclidean ball around $\X^*$, such that for any point $\X$ inside this ball, Eq. \eqref{eq:equalProj} holds.
 
As we show, Proposition \ref{prop:main} readily implies that standard gradient methods such as the \textit{Projected Gradient Method} and \textit{Nesterov's Accelerated Gradient Method}, when initialized in the proximity of an optimal solution, converge with their original convergence guarantees, i.e., \textbf{producing the exact same sequences of iterates}, when the exact Euclidean projection is replaced with the truncated-SVD-based projection $\widehat{\Pi}^r_{\Vert\cdot\Vert_*\leq 1}[\cdot]$.

Some complexity implications of our results to first-order methods for Problem \eqref{eq:optProb}  are summarized in Table \ref{table:compare}, together with comparison to other first-order methods.

The connection between $r$ - the rank parameter in the approximated projection $\widehat{\Pi}_{\Vert\cdot\Vert_*\leq 1}[\cdot]$ and the parameter $\#\sigma_1(\nabla{}f(\X^*))$ may seem unintuitive at first. In particular, one might expect that $r$ should be comparable directly with $\rank(\X^*)$. However, as we show they are indeed tightly related. In particular, the radius of the ball around an optimal solution $\X^*$ in which \eqref{eq:equalProj} holds is strongly related to spectral gaps in the gradient vector $\nabla{}f(\X^*)$. This further implies "over-parameterization" results in which we show how the radius of the ball inside which \eqref{eq:equalProj} applies, increases with the rank parameter $r$, showing it can increase quite dramatically with only a moderate increase in $r$.
We also bring two complementary results showing that $\rank(\X^*) < \#\sigma_1(\nabla{}(\X^*))$ implies that the optimization problem \eqref{eq:optProb} is ill-posed in a sense, and that in general, a result in the spirit of Proposition \ref{prop:main} may not hold when $r < \#\sigma_1(\nabla{}f(\X^*))$.

 \begin{table*}\renewcommand{\arraystretch}{1.3}
{\footnotesize
\begin{center}
  \begin{tabular}{| l | c | c | c | c |}    \hline
    Algorithm  &  Conv. & Rate &  SVD size & Max iterates rank\\\hline
    \multicolumn{5}{|c|}{$\downarrow$~~  $\beta$-smooth  and convex $f$ ~~$\downarrow$} \\ \hline
 Projected Gradient  & global  & $\beta/\epsilon$ & $d$ & $d$ \\\hline
 Accelerated Gradient  & global  & $\sqrt{\beta/\epsilon}$ & $d$ & $d$ \\\hline
 Frank-Wolfe \cite{Jaggi10}  & global  & $\beta/\epsilon$ & $1$ & $\min\{t,d\}$ \\\hline
 Proj. Grad. (this paper) & local  & $\beta/\epsilon$ & $\#\sigma_1(\nabla{}f(\X^*))$ & $\#\sigma_1(\nabla{}f(\X^*))$ \\\hline
 Acc. Grad. (this paper) & local  & $\sqrt{\beta/\epsilon}$ & $\#\sigma_1(\nabla{}f(\X^*))$ & $\#\sigma_1(\nabla{}f(\X^*))$ \\\hline
    \multicolumn{5}{|c|}{$\downarrow$~~  $\beta$-smooth  and $\alpha$-strongly convex $f$~~$\downarrow$} \\ \hline
 Projected Gradident  & global  & $(\beta/\alpha)\log{1/\epsilon}$ & $d$ & $d$\\\hline
  Accelerated Gradient &  global  & $\sqrt{\beta/\alpha}\log{1/\epsilon}$ & $d$ & $d$\\\hline
 ROR-FW  \cite{Garber16a} & global  & $\frac{\beta}{\sigma_{\min}(\X^*)\sqrt{\alpha\epsilon}}$ & $1$ & $\min\{t,d\}$ \\\hline
 BlockFW \cite{allen2017linear} & global  & $(\beta/\alpha)\log{1/\epsilon}$ & $\rank(\X^*)$ & $\min\{\rank(\X^*)(\beta/\alpha),d\}$ \\\hline
  Proj. Grad. (this paper)  & local  & $(\beta/\alpha)\log{1/\epsilon}$ & $\#\sigma_1(\nabla{}f(\X^*))$& $\#\sigma_1(\nabla{}f(\X^*))$\\\hline
 Acc. Grad. (this paper) & local  & $\sqrt{\beta/\alpha}\log{1/\epsilon}$ & $\#\sigma_1(\nabla{}f(\X^*))$ & $\#\sigma_1(\nabla{}f(\X^*))$ \\\hline 
  \end{tabular}
  \caption{Comparison of first-order methods for solving Problem \eqref{eq:optProb}. The 2nd column (from the left) states the type of convergence (either from arbitrary initialization or from a "warm-start"),  the 3rd column states the number iterations to reach $\epsilon$ accuracy, the 4th column states and upper bound on the rank of SVD required on each iteration, and the last column states an upper-bound on the rank of iterates produced by the method.
  }\label{table:compare}
\end{center}
}
\end{table*}\renewcommand{\arraystretch}{1}

\subsection{Organization of this paper}

The rest of this paper is organized as follows. In the remaining of this section we discuss related work. In Section \ref{sec:traceBall} we present our main result: we formalize and prove Proposition \ref{prop:main} in the context of Problem \eqref{eq:optProb}. In this section we also present several complementing results that further strengthen our claims. In Section \ref{sec:traceBallAlgs} we demonstrate how the results of Section \ref{sec:traceBall} readily imply the local convergence of standard projection-based first-order methods for Problem \eqref{eq:optProb}, using only low-rank SVD to compute the Euclidean projection. In Sections \ref{sec:traceReg} and \ref{sec:SDP} we formalize and prove versions of Proposition \ref{prop:main} for smooth convex optimization with trace-norm regularization, and smooth convex optimization over the set of unit-trace positive semidefinite matrices, respectively. Finally, in Section \ref{sec:empirics} we present supporting empirical evidence.

\subsection{Related work}
The subject of efficient algorithms for low-rank matrix optimization problems has enjoyed significant interest in recent years. Below we survey some notable results both for the convex problem \eqref{eq:optProb}, as-well as other related convex models, and also for related non-convex optimization problems. 

\paragraph*{Convex methods:} 
Besides projection-based methods, other highly popular methods for Problem \eqref{eq:optProb} are conditional gradient methods (aka Frank-Wolfe algorithm) \cite{FrankWolfe, Jaggi13b, Jaggi10, Hazan08}. These algorithms require only a rank-one SVD computation on each iteration, hence each iteration is very efficient, however their convergence rates, which are typically of the form $O(1/t)$ for smooth problems (even when the objective is also strongly convex) are in general inferior to projection-based methods such as Nesterov's accelerated method \cite{Nesterov13} and FISTA \cite{FISTA}. Recently, several works have developed variants of the basic method with faster rates, though these hold only under the additional assumption that the objective is also strongly convex  \cite{Garber16a, allen2017linear, garber2018fast}. Additionally, these new variants require to store in memory potentially high-rank matrices, which may limit their applicability to large problems. In \cite{yurtsever2017sketchy} the authors present a novel conditional gradient method which enjoys a low-memory footprint for certain instances of \eqref{eq:optProb} such as the well known matrix completion problem, however there is no improvement in convergence rate beyond that of the standard method.

Besides first-order conditional gradient-type methods, in \cite{MishraMBS13} the authors present a second-order trust-region algorithm for the trace norm-regularized variant of \eqref{eq:optProb}.

\paragraph*{Nonconvex methods:} Problem \eqref{eq:optProb} is often considered as a convex relaxation to the non-convex problem of minimizing $f$ under an explicit rank constraint. Two popular approaches to solving this non-convex problem are i) apply projected gradient descent to the rank-constrained formulation, in which case the projection is onto the set of low-rank matrices, and ii) incorporating the rank-constraint in the objective by considering the factorized objective $g(\U,\V):=f(\U\V^{\top})$, where $\U,\V$ are $m\times r$ and $n\times r$ respectively, where $r$ is an upper-bound on the rank, but otherwise unconstrained. Obtaining global convergence guarantees for these non-convex optimization problems is a research direction of significant interest in recent years, however efficient algorithms are obtained usually only under specific statistical assumptions on the data, which we do not make in this current work, see for instance \cite{Prateek10, jain2014iterative, chen2015fast, bhojanapalli2016global, ge2016matrix} and references therein.

In the works \cite{Dropping16, park2018finding} the authors consider first-order methods for  factorized formulations of problems related to \eqref{eq:optProb}, which are not based on statistical assumptions. In these works the authors establish the convergence of specific algorithms from a good initialization point to the global low-rank optimum with convergence rates similar to that of the standard projected gradient descent method.

\section{Optimization over the Unit Trace-Norm Ball}\label{sec:traceBall}
We begin with introducing some notation. For a positive integer $n$ we let $[n]$ denote the set $\{1,2,\dots,n\}$.
We let $\A\bullet\B$ denote the standard inner product for matrices, i.e., $\A\bullet\B = \trace(\A^{\top}\B)$. For a real matrix $\A$, we let $\sigma_i(\A)$ denote its $i$th largest singular value (including multiplicities), and we let $\#\sigma_i(\A)$ denote the multiplicity of the $i$th largest singular value. Similarly, for a real symmetric matrix $\A$, we let $\lambda_i(\A)$ denote its $i$th largest (signed) eigenvalue, and we let $\#\lambda_i(\A)$ denote the multiplicity of the $i$the largest eigenvalue.
We denote by $\mX^*$ the set of optimal solutions to Problem \eqref{eq:optProb}, and by $f^*$ the corresponding optimal value. %For any feasible $\X$, we define the distance $\dist(\X,\mX^*) := \min_{\X^*\in\mX^*}\Vert{\X-\X^*}\Vert_F$. 

For any $\X\in\reals^{m\times n}$, step-size $\eta$ and radius $\tau$ we denote the projected gradient mapping w.r.t. the trace-norm ball of radius $\tau$:
\begin{eqnarray*}
\projgrad_{\eta,\tau}(\X) := \Pi_{\Vert\cdot\Vert_*\leq\tau}\left[{\X-\eta\nabla{}f(\X)}\right].
\end{eqnarray*}
When $\tau=1$, i.e., we consider the unit trace-norm ball, we will omit the $\tau$ subscript and simply write $\projgrad_{\eta}$.

Given an optimal solution $\X^*\in\mX^*$, a step-size $\eta >0$, and an integer $r$ in  $\{\rank(\X^*),\dots\min\{m,n\}\}$, we let $\delta(\X^*,\eta,r)$ denote the radius of the largest Euclidean ball centered at $\X^*$, such that for all $\X$ in the ball it holds that $\rank\left({\projgrad_{\eta}(\X)}\right) \leq r$. Or equivalently, $\delta(\X^*,\eta,r)$ is the solution to the optimization problem
\begin{eqnarray*}
\sup\delta \geq 0 \quad \textrm{s.t.} \quad \forall\X\in\ball(\X^*,\delta):\rank\left({\projgrad_{\eta}(\X)}\right)\leq r,
\end{eqnarray*}
where $\ball(\X,R)$ denotes the Euclidean ball of radius $R$ centered at $\X$.

%Similarly, for any $\eta >0$ and $r$ in the range $\max\{\rank(\X^*) ~|~\X^*\in\mX^*\},\dots,\min\{m,n\}$, we also define
%\begin{eqnarray*}
%\delta(\mX^*,\eta,r) = \max\{\delta(\X^*,\eta,r) ~|~\X^*\in\mX^*\}.
%\end{eqnarray*}

Towards formalizing and proving Proposition \ref{prop:main}, deriving lower-bounds on the radius $\delta(\X^*,\eta,r)$ will be our main interest in this section.

Since our objective is to study the properties of the projected-gradient mapping over the trace-norm ball, we begin with the following well-known lemma which connects between the SVD of the point to project and the resulting projection.

\begin{lemma}[projection onto the trace-norm ball]\label{lem:traceNormProj}
Fix a parameter $\tau >0$. Let $\X\in\reals^{m\times n}$ and consider its singular-value decomposition $\X = \sum_{i=1}^{\min\{m,n]}\sigma_i\u_i\v_i^{\top}$. If $\Vert{\X}\Vert_* \geq \tau$, then the Euclidean projection of $\X$ onto the trace-norm ball of radius $\tau$ is given by
\begin{eqnarray}\label{eq:traceBallProj}
\Pi_{\Vert{\cdot}\Vert_*\leq \tau}[\X] = \sum_{i=1}^{\min\{m,n\}}\max\{0,\sigma_i-\sigma\}\u_i\v_i^{\top},
%\Pi_{\Vert{\cdot}\Vert_*\leq 1}[\X] = \sum_{i=1}^{\min\{m,n\}}\max\{0,\sigma_i-\sigma\}\u_i\v_i^{\top},
\end{eqnarray}
where $\sigma \geq 0$ is the unique solution to the equation $\sum_{i=1}^{\min\{m,n\}}\max\{0,\sigma_i-\sigma\} = \tau$.
%where $\sigma \geq 0$ is the unique solution to the equation $\sum_{i=1}^{\min\{m,n\}}\max\{0,\sigma_i-\sigma\} = 1$.

Moreover, there exists $r\in\{1,\dots,\min\{m,n\}-1\}$ such that $\sum_{i=1}^r\sigma_i \geq \tau+r\sigma_{r+1}$ if and only if $\rank\left({\Pi_{\Vert\cdot\Vert_*\leq 1}[\X]}\right) \leq r$. 
\end{lemma}
\begin{proof}
The first part of the lemma is a well-known fact, see for instance \cite{Beck17}.
The second part of the lemma comes from the simple observation, that if $\sum_{i=1}^r\sigma_i \geq \tau+r\sigma_{r+1}$ for some $r$, then $\sigma$, as defined in the lemma, must satisfy $\sigma \geq \sigma_{r+1}$, in which case Eq. \eqref{eq:traceBallProj} sets all bottom $(\min\{m,n\}-r)$ components of the SVD of $\X$ to zero, and hence the projection is of rank at most $r$. It is not hard to show using the same reasoning that the reversed direction also holds.
\end{proof}

The following lemma which connects between the singular value decomposition of an optimal solution and its corresponding gradient vector, will play an important technical role in our analysis. The proof of the lemma follows essentially from simple optimality conditions.

\begin{lemma}\label{lem:eigsOfOptGrad}
%Assume that $\nabla{}f$ is non-zero over the unit trace-norm ball.
Let $\X^*\in\mX^*$ be any optimal solution and write its singular value decomposition as $\X^* = \sum_{i=1}^r\sigma_i\u_i\v_i^{\top}$. Then, the gradient vector $\nabla{}f(\X^*)$ admits a singular-value decomposition such that the set of pairs of vectors $\{(-\u_i,\v_i)\}_{i=1}^r$ is a set of top singular-vector pairs of $\nabla{}f(\X^*)$ which corresponds to the largest singular value $\sigma_1(\nabla{}f(\X^*))$.
\end{lemma}
\begin{proof}
First, note that if $\nabla{}f(\X^*) = \mathbf{0}$ then the claim holds trivially. Thus, henceforth we consider the case $\nabla{}f(\X^*)\neq\mathbf{0}$.

It suffices to show that for all $i\in\{1,\dots,r\}$ it holds that $-\u_i^{\top}\nabla{}f(\X^*)\v_i = \sigma_1(\nabla{}f(\X^*))$. 

Assume by contradiction that for some $i\in\{1,\dots,r\}$ it holds that  $-\u_i^{\top}\nabla{}f(\X^*)\v_i < \sigma_1(\nabla{}f(\X^*))$. Let $\u,\v$ denote a singular vector pair corresponding to the top singular value $\sigma_1(\nabla{}f(\X^*))$. Observe that for all $\alpha\in(0,\sigma_i]$, the point $\X_{\alpha}:= \X^* + \alpha(-\u\v^{\top}-\u_i\v_i^{\top})$ is a feasible solution to Problem \eqref{eq:optProb}, i.e., $\Vert{\X_{\alpha}}\Vert_*\leq 1$. Moreover, it holds that
\begin{eqnarray*}
(\X_{\alpha}-\X^*)\bullet\nabla{}f(\X^*) &=& \alpha(-\u\v^{\top}-\u_i\v_i^{\top})\bullet\nabla{}f(\X^*) \\
&<& \alpha(-\sigma_1(\nabla{}f(\X^*)) + \sigma_1(\nabla{}f(\X^*))) = 0,
\end{eqnarray*}
which clearly contradicts the optimality of $\X^*$.
\end{proof}

%The following Corollary is a straightforward consequence of Lemma \ref{lem:eigsOfOptGrad}.

\begin{corollary}\label{cor:lowRankOpt}
For any $\X^*\in\mX^*$ it holds that $\rank(\X^*) \leq \#\sigma_1(\nabla{}f(\X^*))$. Moreover, if $\nabla{}f$ is non-zero over $\mX^*$,
it holds that
\begin{eqnarray}\label{eq:maxmin}
\max\{\rank(\X) ~ |~\X\in\mX^*\} \leq \min\{\#\sigma_1(\nabla{}f(\Y)) ~ |~\Y\in\mX^*\}.
%\max\{\rank(\X^*) ~ |~\X^*\in\mX^*\} \leq \min\{\#\sigma_1(\nabla{}f(\X^*) ~ |~\X^*\in\mX^*\}.
\end{eqnarray}
\end{corollary}
\begin{proof}
Lemma \ref{lem:eigsOfOptGrad}  directly implies that for all $\X^*\in\mX^*$ it holds that $\rank(\X^*) \leq \#\sigma_1(\nabla{}f(\X^*))$. 

For the second part of the lemma,  suppose there exist $\X_1^*,\X_2^*\in\mX^*$ such that $\rank(\X_2^*) > \#\sigma_1(\nabla{}f(\X_1^*))$. 
%First, we observe that if $\nabla{}f(\X_1^*) = \mathbf{0}$, using the convexity of $f$ we have that
%\begin{eqnarray*}
%f(\X_1^*) - f(\X_2^*) \leq (\X_1^* - \X_2^*)\bullet\nabla{}f(\X_1^*) = 0
%\end{eqnarray*}
%and hence we arrive at a contradiction.

Since $\nabla{}f(\X_1^*) \neq \mathbf{0}$, it follows from simple optimality conditions that $\Vert{\X_1^*}\Vert_* = 1$, which together with Lemma \ref{lem:eigsOfOptGrad}  implies that $\X_1^*\bullet\nabla{}f(\X_1^*) = -\sigma_1(\nabla{}f(\X^*))$. Moreover, since $\rank(\X_2^*) > \#\sigma_1(\nabla{}f(\X_1^*))$ and $\Vert{\X_2^*}\Vert_*\leq 1$, it follows that $\X_2^*\bullet\nabla{}f(\X_1^*) > -\sigma_1(\X_1^*)$.

Thus, using again the convexity of $f$ we have that
\begin{eqnarray*}
f(\X_1^*) - f(\X_2^*) \leq (\X_1^* - \X_2^*)\bullet\nabla{}f(\X_1^*) < -\sigma_1(\nabla{}f(\X_1^*)) + \sigma_1(\nabla{}f(\X_1^*)) = 0
\end{eqnarray*}
and hence we arrive at a contradiction.
\end{proof}

One may wonder if the reversed inequality to \eqref{eq:maxmin} holds (i.e., the inequality holds with equality). The following simple example shows that in general the inequality can be strict. Consider the following example.
\begin{eqnarray*}
\min_{\Vert{\X}\Vert_*\leq 1}\{f(\X) := \frac{1}{2}\Vert{\X-\A}\Vert_F^2\}, \qquad \A = \textrm{diag}(1+\sigma, \sigma, \dots, \sigma)\in\reals^{m\times n},
\end{eqnarray*}
for some $\sigma > 0$.

Clearly, using Lemma \ref{lem:traceNormProj}, the problem admits a unique optimal rank-one solution solution $\X^* = \matE_{1,1}$, where $\matE_{1,1}$ denotes the $m\times n$ diagonal matrix with only the first entry along the main diagonal is non-zero and equal to 1. However, one can easily observe that $\nabla{}f(\X^*) = \textrm{diag}(-\sigma,\dots,-\sigma)$, meaning $\#\sigma_1(\nabla{}f(\X^*)) = \min\{m,n\}$.

While the above example demonstrates that in general it is possible that $\#\sigma_1(\nabla{}f(\X^*)) >> \rank(\X^*)$ and that, as a result, Proposition \ref{prop:main} may not imply significant computational benefits for Problem \eqref{eq:optProb}, the following lemma shows that such cases always imply that the optimization problem \eqref{eq:optProb} is ill-posed in the following sense: increasing the radius of the trace-norm ball by an arbitrary small amount will cause the projected gradient mapping to map such original low-rank solution to a higher rank matrix, implying  certain instability of low-rank optimal solutions.

\begin{lemma}[gap necessary for stability of rank of optimal solutions]
Suppose there exists $\X^*\in\mX^*$ of rank $r^*$ such that $\nabla{}f(\X^*)\neq 0$, and suppose that $\#\sigma_1(\nabla{}f^*) > r^*$. Then, for any step-size $\eta >0$ and for any $\epsilon$ small enough, it holds that the projected-gradient mapping  at $\X^*$ w.r.t. the trace-norm ball of radius $1+\epsilon$ satisfies 
\begin{eqnarray*}
\rank\left({\projgrad_{\eta,1+\epsilon}(\X^*)}\right) \geq \#\sigma_1(\nabla{}f(\X^*)).
%\rank(\Pi_{\Vert{\cdot}\Vert_*\leq 1+ \epsilon}[\X^*-\eta\nabla{}f(\X^*)]) \geq \#\sigma_1(\nabla{}f(\X^*)).
\end{eqnarray*}
\end{lemma}

\begin{proof}
Fix some $\X^*\in\mX^*$ and denote $\Y^* := \X^* - \eta\nabla{}f(\X^*)$. Using Lemma \ref{lem:eigsOfOptGrad} we have that the singular values of $\Y^*$ are given by 
\begin{eqnarray*}
\forall i\in[r^*]: \quad \sigma_i &=& \sigma_i(\X^*) + \eta\mu_1; \\
 \forall j>r^*: \quad \sigma_j &=& \eta\mu_j,
\end{eqnarray*}

where $\{\mu_i\}_{i\in[\min\{m,n\}]}$ are the singular values of $\nabla{}f(\X^*)$.
Since $\nabla{}f(\X^*) \neq 0$, which implies that $\Vert{\X^*}\Vert_*=1$, it holds that $\Vert{\Y^*}\Vert_* > 1 $. Let $\epsilon>0$ be such that $\Vert{\Y^*}\Vert_* \geq 1+ \epsilon$. Then, by Lemma \ref{lem:traceNormProj}, we have that the projected-gradient mapping w.r.t. the trace-norm ball of radius $(1+\epsilon)$ satisfies:
\begin{eqnarray*}
\projgrad_{\eta,1+\epsilon}(\X^*)=\Pi_{\Vert{\cdot}\Vert_* \leq 1+\epsilon}[\Y^*] = \sum_{i=1}^{\min\{m,n\}}\max\{0,\sigma_i-\sigma\}\u_i\v_i^{\top},
\end{eqnarray*}

where $\sigma$ satisfies: $\sum_{i=1}^{\min\{m,n\}}\max\{0, \sigma_i-\sigma\} = 1+\epsilon$. Observe that for $\sigma = \eta\mu_1$, we have that
\begin{eqnarray*}
\sum_{i=1}^{\min\{m,n\}}\max\{0, \sigma_i-\sigma\} = \sum_{i=1}^{r^*}\sigma_i(\X^*)= 1 < 1+ \epsilon.
\end{eqnarray*}
Thus, it must hold that $\sigma < \eta\mu_1$. However, then it follows that for all $i\in[\#\sigma_1(\nabla{}f(\X^*))]$, $\sigma_i-\sigma >0$ and thus, $\rank\left({\Pi_{\Vert{\cdot}\Vert_*\leq 1+\epsilon}[\Y^*]}\right) \geq \#\sigma_1(\nabla{}f^*)$.
\end{proof}

%The Lemma says that for any optimal solution, in the absence of a spectral gap in corresponding gradient vector, if we increase the radius of the trace-norm ball by an arbitrarily small amount, then all original low-rank solutions become unstable, in the sense, that performing a single projected-gradient step results in a high-rank matrix, when $\#\sigma_1(\nabla{}f(\X^*))$ is considerably larger than $\rank(\X^*)$, indicating the optimization problem \eqref{eq:optProb} is ill-posed in this sense.

The following lemma demonstrates why setting the rank of the truncated-SVD projection to be at least $\#\sigma_1(\nabla{}f(\X^*))$ is necessary. The lemma shows that in general, a result similar in spirit to Proposition \ref{prop:main} may not hold with SVD rank parameter $r$ satisfying $r  < \#\sigma_1(\nabla{}f(\X^*))$.

\begin{lemma}\label{lem:rankLB}
Fix a positive integer $n$ and $r\in\{2,\dots,n-1\}$. Then, for any $a\in(0,1)$ small enough and for any $\sigma >0$, there exists a convex and $1$-smooth function $f:\reals^{n\times n}\rightarrow\reals$ such that
\begin{enumerate}
\item
$f$ admits a rank-$r$ minimizer over the unit trace-norm ball $\X^*$ for which it holds that $\#\sigma_1(\nabla{}f(\X^*)) = r+1$ and the spectral gap is $\sigma_1(\nabla{}f(\X^*))-\sigma_{r+2}(\nabla{}f(\X^*)) = \sigma$,
\item
there exists a matrix $\X_a$ such that $\Vert{\X_a}\Vert_*\leq 1$, $\rank(\X_a) = r$, $\Vert{\X_a-\X^*}\Vert_F \leq \sqrt{2a}$, and for any $\eta\in(0,1]$ it holds that $\rank(\projgrad_{\eta}(\X_a))  = r+1$.
\end{enumerate}

\end{lemma}

\begin{proof}
Consider the following function $f:\reals^{n\times n}\rightarrow\reals$.
\begin{eqnarray*}
f(\X) := \frac{1}{2}\sum_{i=1}^r\left({\X\bullet\matE_{i,i} - (\lambda_i+\sigma)}\right)^2 + \frac{1}{2}\left({\X\bullet\frac{1}{2}\matE_{r+1,r+1}-\sigma}\right)^2,
\end{eqnarray*}

where $\matE_{i,i}$ denotes the indicator for the $i$th diagonal entry. Note that $f$ is indeed $1$-smooth.

We set values $\lambda_i = \frac{1-a}{r-1}, i=1,\dots,r-1$, $\lambda_r = a$.
It is not hard to verify that the rank-$r$ matrix
\begin{eqnarray*}
\X^* = \frac{1-a}{r-1}\sum_{i=1}^{r-1}\matE_{i,i} + a\matE_{r,r}
%\X^* = \textrm{diag}\left({\frac{1-a}{1-r},\dots,\frac{1-a}{1-r},a,0,\dots,0}\right)
\end{eqnarray*}
is a minimizer of $f$ over the unit trace-norm ball. In particular, it holds that
\begin{eqnarray*}
\nabla{}f(\X^*) = -\sigma\sum_{i=1}^{r+1}\matE_{i,i}.
\end{eqnarray*}
Hence, we have $\#\sigma_1(\nabla{}f(\X^*)) =r+1 > \rank(\X^*)$.

Consider now the matrix $\X_a$ given by
\begin{eqnarray*}
\X_a = \frac{1-a}{r-1}\sum_{i=1}^{r-1}\matE_{i,i} + a\matE_{r+1,r+1}.
%\X_a = \textrm{diag}\left({\frac{1-a}{1-r},\dots,\frac{1-a}{1-r},0,a,0,\dots,0}\right).
\end{eqnarray*}
Note that $\X_a$ is rank-$r$ as well. Clearly, it holds that $\Vert{\X^*-\X_a}\Vert_F = \sqrt{2}a$.

Also,
\begin{eqnarray*}
\nabla{}f(\X_a) = -\sigma\sum_{i=1}^{r-1}\matE_{i,i} - (a+\sigma)\matE_{r,r} + (\frac{a}{2}-\sigma)\matE_{r+1,r+1}.
%\nabla{}f(\X_a) = \textrm{diag}\left({-\sigma,\dots,-\sigma,-(a+\sigma),0,\dots,0}\right).
\end{eqnarray*}
Thus, for any step-size $\eta\in(0,1]$ we have
\begin{eqnarray*}
\Y_a := \X_a-\eta\nabla{}f(\X_a) &=& \left({\frac{1-a}{r-1}+\eta\sigma}\right)\sum_{i=1}^{r-1}\matE_{i,i} +\eta(a+\sigma)\matE_{r,r} \\
&&+ (a - \eta\frac{a}{2} + \eta\sigma)\matE_{r+1,r+1}.
\end{eqnarray*}

Note that $\Y_a$ has $r+1$ positive singular values, which we denote (in non-increasing order) by $\gamma_1,\dots,\gamma_{r+1}$. In particular, for any $a \leq 1/r$ it holds that $\gamma_i = (1-a)/(r-1)+\eta\sigma$ for  $i=1,\dots,r-1$.

Note that $\sum_{i=1}^{r+1}\gamma_i > 1$. Thus, by Lemma \ref{lem:traceNormProj}, the singular values of $\projgrad_{\eta}(\X_a)$ are given by $\max\{\gamma_i-\gamma,0\},i=1,\dots,r+1$, where $\gamma >0$ satisfies
\begin{eqnarray*}
\sum_{i=1}^{r+1}\max\{0,\gamma_i-\gamma\} = 1.
\end{eqnarray*}

For $\rank(\projgrad_{\eta}(\X_a)) \leq r$ to hold, it must hold that $\gamma \geq \gamma_{r+1}$. We consider now two cases. 

In the first case we have $\eta(a+\sigma) \geq a(1-\eta/2)+\eta\sigma$, i.e., $\gamma_r = \eta(a+\sigma)$, $\gamma_{r+1} =  a(1-\eta/2)+\eta\sigma$. Then, for $\rank(\projgrad_{\eta}(\X_a)) \leq r$ to hold, it must hold that
\begin{eqnarray*}
1 = \sum_{i=1}^{r+1}\max\{0,\gamma_i-\gamma\} &\leq& \sum_{i=1}^{r+1}\max\{0,\gamma_i-a(1-\eta/2)-\eta\sigma\} \\
&<& (r-1)\frac{1-a}{r-1} + \eta{}a= 1,
\end{eqnarray*}
and hence we arrive at a contradiction.

In the second case we have $\eta(a+\sigma) < a(1-\eta/2)+\eta\sigma$, i.e., $\gamma_r = a(1-\eta/2)+\eta\sigma$, $\gamma_{r+1} =  \eta(a+\sigma)$. As in the first case, in order for $\rank(\projgrad_{\eta}(\X_a)) \leq r$ to hold, it must hold that
\begin{eqnarray*}
1 = \sum_{i=1}^{r+1}\max\{0,\gamma_i-\gamma\} &\leq& \sum_{i=1}^{r+1}\max\{0,\gamma_i-\eta(a+\sigma)\} \\
&<& (r-1)\frac{1-a}{r-1} + a - \eta\frac{3a}{2}< 1,
\end{eqnarray*}
and thus, in this case also we arrive at a contradiction.

We thus conclude that $\rank(\projgrad_{\eta}(\X_a)) =r+1 > r$.
\end{proof}

We now present and prove our main technical theorem which lower bounds $\delta(\X^*,\eta,r)$ - the radius of the ball around an optimal solution $\X^*$ inside which the projected gradient mapping $\projgrad_{\eta}(\cdot)$ has rank at most $r$, hence proving Proposition \ref{prop:main}.

\begin{theorem}\label{thm:goodProj}
Let $f:\reals^{m\times n}\rightarrow\reals$ be $\beta$-smooth and convex. Assume $\nabla{}f$ is non-zero over the unit trace-norm ball and fix some $\X^*\in\mX^*$. Let $r$ denote the multiplicity of $\sigma_1(\nabla{}f(\X^*))$, and let $\mu_1,\mu_2,\dots,\mu_{\min\{m,n\}}$ denote the singular values of $\nabla{}f(\X^*)$ (including multiplicities).  Then, for any $\eta >0$ it holds that
\begin{eqnarray}\label{eq:goodProjRes:1}
\delta(\X^*,\eta,r) \geq  \frac{\eta{}(\mu_1-\mu_{r+1})}{(1+1/\sqrt{r})(1+\eta\beta)}.
\end{eqnarray}

%Then, for any $\eta >0$ and for any $\X$ (not necessarily feasible) satisfying $\Vert{\X-\X^*}\Vert_F \leq \delta(\X^*,\eta,r) :=  \frac{\eta{}(\mu_1-\mu_{r+1})}{(1+1/\sqrt{r})(1+\eta\beta)}$ it holds that $\rank\left({\Pi_{\Vert\cdot\Vert_*\leq1}[\X-\eta\nabla{}f(\X)]}\right) \leq r$.

More generally, for any $\eta>0$ and $r'\in\{r,\dots,\min\{m,n\}-1\}$, it holds that 
\begin{eqnarray}\label{eq:goodProjRes:2}
\delta(\X^*,\eta,r') \geq  \frac{\eta(\mu_1-\mu_{r'+1})}{(1+1/\sqrt{r})(1+\eta\beta)}.
\end{eqnarray}

%for any $\X$ satisfying $\Vert{\X-\X^*}\Vert_F  \leq \delta(\X^*,\eta,r') :=  \frac{\eta(\mu_1-\mu_{r'+1})}{(1+1/\sqrt{r})(1+\eta\beta)}$ that  $\rank\left({\Pi_{\Vert\cdot\Vert_*\leq1}[\X-\eta\nabla{}f(\X)]}\right) \leq r'$.

Moreover, for any $\eta >0$ and $r'\in\{r,\dots,\min\{m,n\}-r\}$, it holds that 
\begin{eqnarray}\label{eq:goodProjRes:3}
\delta(\X^*,\eta,r'+r-1) \geq  \frac{\sqrt{r}\eta(\mu_1-\mu_{r'+1})}{2(1+\eta\beta)}.
\end{eqnarray}

%and any $\eta >0$, it holds for any $\X$ satisfying $\Vert{\X-\X^*}\Vert_F  \leq \frac{\sqrt{r}\eta(\mu_1-\mu_{r'+1})}{2(1+\eta\beta)}$ that  $\rank\left({\Pi_{\Vert\cdot\Vert_*\leq1}[\X-\eta\nabla{}f(\X)]}\right) \leq r'' = r'+r-1$.
\end{theorem}

\begin{proof}
Throughout the proof we assume without loss of generality that $m \leq n$.  Fix a step-size $\eta>0$. 

Denote $\Y^* := \X^*-\eta\nabla{}f(\X^*)$ and let $\sigma_1,\dots,\sigma_m$ denote the singular values of $\Y^*$. Let us also denote by $\mu_1,\dots\mu_m$ the singular values of $\nabla{}f(\X^*)$, and  $r^* := \rank(\X^*)$.
From Lemma \ref{lem:eigsOfOptGrad} we can deduce that
\begin{eqnarray}\label{eq:goodProj:0}
\forall i\in[r^*]: \quad \sigma_i &=& \sigma_i(\X^*) + \eta\mu_1; \nonumber \\
 \forall j>r^*: \quad \sigma_j &=& \eta\mu_j.
\end{eqnarray}

For any integer $r'\in\{r,\dots,m\}$ let us define 
\begin{eqnarray*}
\xi(r') = \sum_{i=1}^{r}\sigma_i - r\cdot\sigma_{r'+1} -1.
\end{eqnarray*}

Since $\nabla{}f(\X^*)\neq\mathbf{0}$, it follows that $\sum_{i=1}^{r^*}\sigma_i(\X^*) =\sum_{i=1}^{r}\sigma_i(\X^*) = \Vert{\X^*}\Vert_*=1$, we have that

\begin{eqnarray}\label{eq:goodProj:1}
\xi(r') := \sum_{i=1}^{r}\sigma_i - r\cdot\sigma_{r'+1} -1 \underset{(a)}{=} 1+\eta{}r\mu_1  - r\eta\mu_{r'+1}  - 1= \eta{}r(\mu_1-\mu_{r'+1}),
\end{eqnarray}
where (a) follows from \eqref{eq:goodProj:0}.

Now, given some $\X\in\reals^{m\times n}$, denote $\Y := \X-\eta\nabla{}f(\X)$ and let $\gamma_1,\dots\gamma_m$ denote the singular values of $\Y$. It holds that

\begin{eqnarray}\label{eq:goodProj:2}
\sum_{i=1}^{r}\gamma_i &\underset{(a)}{\geq} &\sum_{i=1}^{r}\sigma_i - \sum_{i=1}^{r}\sigma_i(\Y-\Y^*) \geq \sum_{i=1}^{r}\sigma_i -\sqrt{r\sum_{i=1}^{r}\sigma_i^2(\Y-\Y^*)} \nonumber \\
&\geq& \sum_{i=1}^{r}\sigma_i -\sqrt{r\sum_{i=1}^{m}\sigma_i^2(\Y-\Y^*)} = \sum_{i=1}^{r}\sigma_i  -\sqrt{r}\Vert{\Y-\Y^*}\Vert_F \nonumber \\
&=& \sum_{i=1}^{r}\sigma_i - \sqrt{r}\Vert{\X-\eta\nabla{}f(\X) - \X^* + \eta\nabla{}f(\X^*)}\Vert_F \nonumber \\
&\geq& \sum_{i=1}^{r}\sigma_i - \sqrt{r}\left({\Vert{\X-\X^*}\Vert_F + \eta\Vert{\nabla{}f(\X) - \nabla{}f(\X^*)}\Vert_F}\right) \nonumber \\
&\underset{(b)}{\geq}& \sum_{i=1}^{r}\sigma_i - \sqrt{r}(1+\eta\beta)\Vert{\X-\X^*}\Vert_F,
\end{eqnarray}
where (a) follows from Ky Fan's inequality for the singular values, and (b) follows from the $\beta$-smoothness of $f(\cdot)$.

Also, similarly, using Weyl's inequality, it holds that
\begin{eqnarray}\label{eq:goodProj:3}
\gamma_{r'+1} \leq \sigma_{r'+1} + \sigma_1(\Y-\Y^*) \leq \sigma_{r'+1} + (1+\eta\beta)\Vert{\X-\X^*}\Vert_F.
\end{eqnarray}

Combining Eq. \eqref{eq:goodProj:1}, \eqref{eq:goodProj:2}, \eqref{eq:goodProj:3}, we have that
\begin{eqnarray}\label{eq:goodProj:4}
\sum_{i=1}^{r'}\gamma_i - r'\gamma_{r'+1} &\geq & \sum_{i=1}^r\gamma_i - r\gamma_{r'+1} \nonumber \\
& \geq &
\sum_{i=1}^{r}\sigma_i - \sqrt{r}(1+\eta\beta)\Vert{\X-\X^*}\Vert_F - r\left({\sigma_{r'+1} + (1+\eta\beta)\Vert{\X-\X^*}\Vert_F}\right) \nonumber\\
&=& \sum_{i=1}^{r}(\sigma_i-\sigma_{r'+1}) -(r+\sqrt{r})(1+\eta\beta)\Vert{\X-\X^*}\Vert_F \nonumber\\
&=& 1 + \xi(r') - (r+\sqrt{r})(1+\eta\beta)\Vert{\X-\X^*}\Vert_F.
\end{eqnarray}

Thus, it follows that if $\X$ satisfies:
\begin{eqnarray*}%\label{eq:goodProj:5}
\Vert{\X-\X^*}\Vert_F \leq  \frac{\xi(r')}{(r+\sqrt{r})(1+\eta\beta)} = \frac{\eta(\mu_1-\mu_{r'+1})}{(1+1/\sqrt{r})(1+\eta\beta)},
\end{eqnarray*}
we have that $\sum_{i=1}^{r'}\gamma_i - r'\gamma_{r'+1} \geq 1$, which implies via Lemma \ref{lem:traceNormProj} that $\rank\left({\projgrad_{\eta}(\X)}\right)\leq r$. This proves \eqref{eq:goodProjRes:1}, \eqref{eq:goodProjRes:2}.

Alternatively, for any $r'' \geq r' + r-1$, using the more general version of Weyl's inequality, we can replace Eq. \eqref{eq:goodProj:3} with
\begin{eqnarray}\label{eq:goodProj:6}
\gamma_{r''+1} &\leq & \sigma_{r'+1} + \sigma_{r''-r'+1}(\Y-\Y^*) = \sigma_{r'+1} + \sqrt{\sigma^2_{r''-r'+1}(\Y-\Y^*)} \nonumber \\
& \leq & \sigma_{r'+1} + \sqrt{\frac{1}{r''-r'+1}\Vert{\Y-\Y^*}\Vert_F^2} \nonumber \\
&\leq &\sigma_{r'+1} + \frac{1}{\sqrt{r''-r'+1}}(1+\eta\beta)\Vert{\X-\X^*}\Vert_F.
\end{eqnarray}

Thus, similarly to Eq. \eqref{eq:goodProj:4}, but replacing Eq. \eqref{eq:goodProj:3} with Eq. \eqref{eq:goodProj:4}, we obtain
\begin{eqnarray*}
\sum_{i=1}^{r''}\gamma_i - r''\gamma_{r''+1} &\geq &\sum_{i=1}^{r}\gamma_i - r\gamma_{r''+1}\geq \sum_{i=1}^{r}\sigma_i  - \sqrt{r}(1+\eta\beta)\Vert{\X-\X^*}\Vert_F \nonumber \\
&&- r\left({\sigma_{r'+1} + \frac{1}{\sqrt{r''-r'+1}}(1+\eta\beta)\Vert{\X-\X^*}\Vert_F}\right) \nonumber \\
&=& \sum_{i=1}^{r}(\sigma_i-\sigma_{r'+1}) - \left({\sqrt{r} + \frac{r}{\sqrt{r''-r'+1}}}\right)(1+\eta\beta)\Vert{\X-\X^*}\Vert_F \nonumber \\
&=& 1 + \xi(r') - \left({\sqrt{r} + \frac{r}{\sqrt{r''-r'+1}}}\right)(1+\eta\beta)\Vert{\X-\X^*}\Vert_F.
\end{eqnarray*}

In particular, for $r'' = r'+r-1$, we obtain
\begin{eqnarray*}
\sum_{i=1}^{r''}\gamma_i - r''\gamma_{r''+1} \geq 1 + \xi(r') - 2\sqrt{r}(1+\eta\beta)\Vert{\X-\X^*}\Vert_F.
\end{eqnarray*}

Thus, it follows that if $\X$ satisfies:
\begin{eqnarray*}
\Vert{\X-\X^*}\Vert_F \leq \frac{\xi(r')}{2\sqrt{r}(1+\eta\beta)} = \frac{r\eta(\mu_1-\mu_{r'+1})}{2\sqrt{r}(1+\eta\beta)} = \frac{\sqrt{r}\eta(\mu_1-\mu_{r'+1})}{2(1+\eta\beta)},
\end{eqnarray*}
we have that  $\rank\left({\projgrad_{\eta}(\X)}\right)\leq r''$, which proves \eqref{eq:goodProjRes:3}.

\end{proof}

Note that the last two parts of Theorem \ref{thm:goodProj}, i.e., Eq. \eqref{eq:goodProjRes:2}, \eqref{eq:goodProjRes:3}, provide strong "over-parameterization" results, showing that, depending on the singular values of $\nabla{}f(\X^*)$, the radius $\delta(\X^*,\eta,r)$ can increase dramatically by taking the rank parameter $r$ to be only moderately larger than $\#\sigma_1(\nabla{}f(\X^*))$.

The following theorem complements Theorem \ref{thm:goodProj}, stating that the estimate \eqref{eq:goodProjRes:1} on $\delta(\X^*,\eta,r)$ is tight up to a small universal constant.

\begin{theorem}\label{thm:lowerBound}
Fix $\beta >0$, $r\in\{2,\dots,\min\{m,n\}-1\}$, a real scalar $\sigma >0$, and $\eta\in(0,1/\beta]$ such that $\eta\sigma <1$. For any $\epsilon >0$ small enough,  there exists a convex and $\beta$-smooth function $f:\reals^{m\times n}\rightarrow\reals$ which admits a rank-$r$ minimizer over the unit trace-norm ball $\X^*$ for which $\#\sigma_1(\nabla{}f(\X^*)) = r$,  $\sigma_1(\nabla{}f(\X^*)) -\sigma_{r+1}(\nabla{}f(\X^*))=\sigma$, and %Moreover, for any $\eta\in(0,1/\beta]$ such that $\eta\sigma <1$ it holds that
\begin{eqnarray*}
\delta(\X^*,\eta,r) < \sqrt{2}\eta\sigma + \epsilon.
\end{eqnarray*}
%and for any $\epsilon>0$ small enough, there exists a rank-$r$ matrix $\X$, $\Vert{\X}\Vert_*\leq 1$ such that
%\begin{eqnarray*}
%\Vert{\X-\X^*}\Vert_F = \sqrt{2}\eta\gap(\nabla{}f(\X^*))+\epsilon \quad \textrm{and} \quad \rank\left({\projgrad_{\eta}(\X)}\right) > r.
%\end{eqnarray*}
\end{theorem}

\begin{proof}
We use a construction similar to the one used in the proof of Lemma \ref{lem:rankLB}. 
Consider the following function $f:\reals^{n\times n}\rightarrow\reals$.
\begin{eqnarray*}
f(\X) := \frac{\beta}{2}\sum_{i=1}^r\left({\X\bullet\matE_{ii} - (\lambda_i+\sigma/\beta)}\right)^2,
\end{eqnarray*}

where $\matE_{i,i}$ denotes the indicator for the $i$th diagonal entry.

We set values $\lambda_i = \frac{1-a}{r-1}, i=1,\dots,r-1$, and $\lambda_r = a$, for some $a\in(0,1)$ to be determined later. %, and we assume $\sigma >0$. 
It is not hard to verify that 
\begin{eqnarray*}
\X^* = \frac{1-a}{r-1}\sum_{i=1}^{r-1}\matE_{i,i} + a\matE_{r,r}
%\X^* = \textrm{diag}\left({\frac{1-a}{1-r},\dots,\frac{1-a}{1-r},a,0,\dots,0}\right)
\end{eqnarray*}
is a minimizer of $f$ over the unit trace-norm ball. In particular, it holds that
\begin{eqnarray*}
\nabla{}f(\X^*) = -\sigma\sum_{i=1}^r\matE_{i,i}.
%\nabla{}f(\X^*) = \textrm{diag}\left({-\sigma,\dots,-\sigma,0,\dots,0}\right).
\end{eqnarray*}
Hence we have $\sigma_1(\nabla{}f(\X^*))-\sigma_{r+1}(\nabla{}f(\X^*)) = \sigma$.

Consider now the point $\X_a$ given by
\begin{eqnarray*}
\X_a = \frac{1-a}{r-1}\sum_{i=1}^{r-1}\matE_{i,i} + a\matE_{(r+1),(r+1)}.
%\X_a = \textrm{diag}\left({\frac{1-a}{1-r},\dots,\frac{1-a}{1-r},0,a,0,\dots,0}\right).
\end{eqnarray*}
Note that $\X_a$ is rank-$r$ as well. Clearly, it holds that $\Vert{\X^*-\X_a}\Vert_F = \sqrt{2}a$.

Also,
\begin{eqnarray*}
\nabla{}f(\X_a) = -\sigma\sum_{i=1}^{r-1}\matE_{i,i} -(\beta{}a+\sigma)\matE_{r,r}.
%\nabla{}f(\X_a) = \textrm{diag}\left({-\sigma,\dots,-\sigma,-(a+\sigma),0,\dots,0}\right).
\end{eqnarray*}
Thus, for any step-size $\eta\in(0,1/\beta]$ we have
\begin{eqnarray*}
\X_a-\eta\nabla{}f(\X_a) = \left({\frac{1-a}{r-1}+\eta\sigma}\right)\sum_{i=1}^{r-1}\matE_{i,i} +\eta(\beta{}a+\sigma)\matE_{r,r} + a\matE_{(r+1),(r+1)}.
%\X_a-\nabla{}f(\X_a) =\textrm{diag}\left({\frac{1-a}{1-r}+\eta\sigma,\dots,\frac{1-a}{1-r}+\eta\sigma,\eta(a+\sigma),a,0,\dots,0}\right).
\end{eqnarray*}

We now show that if $\rank\left({\projgrad_{\eta}(\X_a)}\right) \leq r$, then it must hold that  $a\leq\eta\sigma$. To see this, assume by contradiction that $a > \eta\sigma$. We consider two cases. First, if $a \leq \eta\beta{}a + \eta\sigma$, then denoting the $r+1$ non-zero singular values of $\Y_a := \X_a - \eta\nabla{}f(\X_a)$ by $\sigma_1,\dots,\sigma_{r+1}$, we have that $\sigma_{r+1} = a$. Thus, according to Lemma \ref{lem:traceNormProj}, in order for $\rank\left({\projgrad_{\eta}(\X_a)}\right) \leq r$ to hold, it must hold that

\begin{eqnarray*}
1 &\leq &\sum_{i=1}^r\sigma_i - r\sigma_{r+1} = (1-a) + (r-1)\eta\sigma +  \eta\beta{}a + \eta\sigma - r{}a  \\
&=& 1 - (r+1)a + r\eta\sigma{} + \eta\beta{}a \underset{(a)}{\leq} 1 - r(a-\eta\sigma) \underset{(b)}{<} 1,
\end{eqnarray*}
where (a) follows from our assumption that $\eta\in(0,1/\beta]$ and (b) follows from the assumption $a > \eta\sigma$. Hence, we arrive at a contradiction.

In the second case, we assume $a > \eta\beta{}a + \eta\sigma$. Now, the smallest non-zero singular value of $\Y_a$ is $\sigma_{r+1} = \eta\beta{}a + \eta\sigma$. As before, using Lemma \ref{lem:traceNormProj}, now, in order for $\rank\left({\projgrad_{\eta}(\X_a)}\right) \leq r$ to hold, it must hold that
\begin{eqnarray*}
1 &\leq &\sum_{i=1}^r\sigma_i - r\sigma_{r+1} = (1-a) + (r-1)\eta\sigma  + a - r(\eta\beta{}a + \eta\sigma) \\
&=& 1- \eta\sigma - r\eta\beta{}a < 1.
\end{eqnarray*}
Hence, in this case we also arrive at a contradiction.

Thus, we conclude that
\begin{eqnarray*}
\forall \eta\in(0,1/\beta]: \quad \rank\left({\projgrad_{\eta}(\X_a)}\right) \leq r ~ \Longrightarrow ~ \eta\sigma \geq a.
\end{eqnarray*}
%or equivalently,
%\begin{eqnarray*}
%\forall \eta\in(0,1/\beta]: \quad \rank\left({\projgrad_{\eta}(\X_a)}\right) \leq r ~ \Longrightarrow ~ \sqrt{2}\eta\cdot\gap \geq \Vert{\X^*-\X_a}\Vert_F.
%\end{eqnarray*}

Thus, for any $\epsilon >0$ small enough, setting $a = \eta\sigma+\epsilon/\sqrt{2}$ we have that $\Vert{\X_a-\X^*}\Vert_F = \sqrt{2}\eta\sigma+\epsilon$, and that $\rank(\projgrad_{\eta}(\X_a)) > r$.

\end{proof}

Since inside the ball of radius $\delta(\X^*,r,\eta)$ around $\X^*$ all iterates of a projected gradient-based method are of rank at most $r$, one may wonder if Theorem \ref{thm:lowerBound} still holds if we restrict our attention to the intersection of this ball with the set of all matrices with rank at most $r$. The answer is yes, since as we see from the proof, the constructed "bad" matrix $\X_a$ is of rank $r$ as well.

%\subsection{Some specific functions}
%\begin{lemma}
%Suppose that $f(\X) = g(\mA(\X))$, where $g$ is smooth and strictly convex. Then the gradient vector of $f$ is constant over the optimal set $\mX^*$.
%\end{lemma}
%\begin{proof}
%In order to prove that the gradient vector is constant over $\mX^*$, note it suffices to show that $\mA(\X)$ is constant over $\mX^*$. To see this, assume by contradiction that there exists $\X_1^*,\X_2^*\in\mX^*$ such that $\mA(\X_1^*) \neq \mA(\X_2^*)$. Using the strict convexity of $g(\cdot)$, we now have that
%\begin{eqnarray*}
%f^* &=& f\left({\frac{\X_1^*+\X_2^*}{2}}\right) = g\left({\mA\left({\frac{\X_1^*+\X_2^*}{2}}\right)}\right) = g\left({\frac{\mA(\X_1^*)+\mA(\X_2^*)}{2}}\right) \\
%&< &\frac{g(\mA(\X_1^*)) + g(\mA(\X_2^*))}{2} = f^*.
%\end{eqnarray*}
%Hence, we arrive at a contradiction.
%\end{proof}

\section{Local convergence results with low-rank projections}\label{sec:traceBallAlgs}

In this section we discuss concrete algorithmic implications of Theorem \ref{thm:goodProj} to the local convergence of first-order methods for solving Problem \eqref{eq:optProb}. We demonstrate the immediate applicability of our results to the local convergence of the projected gradient descent method and Nesterov's accelerated gradient method \cite{Nesterov12}. At the end of this section we also partially discuss implications for the stochastic gradient descent method \cite{Bubeck15}.

Throughout this section we assume the following assumption holds true.
\begin{assumption}
For all $\X^*\in\mX^*$ it holds that $\nabla{}f(\X^*) \neq \mathbf{0}$ (which in turn implies that for all $\X^*\in\mX^*$, $\Vert{\X^*}\Vert_*=1$).
\end{assumption}

\begin{theorem}\label{thm:PGD}[local convergence of Gradient Descent]
Consider the sequence $\{\X_t\}_{t\geq 0}$ produced by the Projected Gradient Method (PGD):
\begin{eqnarray}\label{eq:thm:PGD:pgd}
\X_0\in\{\X\in\reals^{m\times n} ~|~\Vert{\X}\Vert_*\leq 1\}, \quad \forall t\geq 0: ~ \X_{t+1} \gets \projgrad_{1/\beta}(\X_t).
%\X_0\in\{\X ~|~\Vert{\X}\Vert_*\leq 1\}, \quad \forall t\geq 0: ~ \X_{t+1} \gets \Pi_{\Vert{\cdot}\Vert_*\leq 1}\left[{\X_{t} - \frac{1}{\beta}\nabla{}f(\X_{t})}\right].
\end{eqnarray}
For any $\X^*\in\mX^*$ and any $r \geq\#\sigma_1\left({\nabla{}f(\X^*)}\right)$, it holds that if PGD is initialized with a feasible $\X_0$ which satisfies $\Vert{\X-\X^*}\Vert_F \leq \delta(\X^*,r,1/\beta)$, then, replacing the projection $\Pi_{\Vert{\cdot}\Vert_*\leq 1}[\cdot]$ in the mapping $\projgrad_{1/\beta}$ with the rank-$r$ projection $\hat{\Pi}^r_{\Vert{\cdot}\Vert_*\leq 1}[\cdot]$ does not change the sequence $\{\X_t\}_{t\geq 0}$.

In particular, the standard convergence guarantees on the error $f(\X_t)-f^*$ ($O(\beta\Vert{\X_0-\X^*}\Vert_F^2/t)$ for convex $f$ and $O\left({\exp(-\Theta(\alpha/\beta)t)}\right)$ for $\alpha$-strongly convex $f$ \cite{Nesterov12}) hold.
\end{theorem}

\begin{proof}
The proof is by simple induction. By definition of $\delta(\X^*,r,1/\beta)$, using the short notation $\delta = \delta(\X^*,r,1/\beta)$, we have that if $\X_t\in\ball(\X^*,\delta)$ then $\rank(\projgrad_{1/\beta}(\X_t)) \leq r$, which means the accurate projection can be replaced with the rank-$r$ truncated projection  $\hat{\Pi}^r_{\Vert{\cdot}\Vert_*\leq 1}[\cdot]$. Thus, it remains to show that if $\X_0\in\ball(\X^*,\delta)$ then for all $t\geq 1$, $\X_t\in\ball(\X^*,\delta)$ holds as well. However, the latter is a well known fact. Indeed for the projected gradient mapping, for any feasible $\X$ and for any optimal solution $\X^*$ it holds that $\Vert{\projgrad_{1/\beta}(\X)-\X^*}\Vert_F \leq \Vert{\X-\X^*}\Vert_F$,  see for instance \cite{Bubeck15} (proof of Theorem 3.7). Thus, the result holds.
\end{proof}

\begin{theorem}\label{thm:AGD_SC}[local convergence of Accelerated Gradient Method for strongly convex $f$]
Suppose $f$ is $\alpha$-strongly convex with $\alpha < \beta$.
Consider Nesterov's Accelerated Gradient Method for smooth and strongly convex minimization \cite{Nesterov13}, given by the update rule:
\begin{eqnarray}\label{eq:thm:AGD_SC:agd}
%\forall t\geq 0: \qquad \X_{t+1}  & \gets &\Pi_{\Vert{\cdot}\Vert_*\leq 1}\left[{\Y_{t} - \frac{1}{\beta}\nabla{}f(\Y_{t})}\right] \\
\forall t\geq 0: \qquad \X_{t+1}  & \gets &\projgrad_{1/\beta}(\Y_t) \\
\Y_{t+1}  & \gets & \X_{t+1} + \frac{\sqrt{\beta}-\sqrt{\alpha}}{\sqrt{\beta}+\sqrt{\alpha}}(\X_{t+1} - \X_t), \nonumber 
\end{eqnarray}

where $\Y_0=\X_0\in\{\X ~|~\Vert{\X}\Vert_*\leq 1\}$. Let $\X^*$ denote the unique optimal solution and let $r\in\{\#\sigma_1(\nabla{}f(\X^*)),\dots,\min\{m,n\}\}$. Then, if $\Vert{\X_0-\X^*}\Vert_F \leq \frac{\sqrt{\alpha}}{3\sqrt{\alpha+\beta}}\delta(\X^*,r,1/\beta)$, we have that replacing the projection $\Pi_{\Vert\cdot\Vert_*\leq 1}[\cdot]$ in Eq. \eqref{eq:thm:AGD_SC:agd} with the rank-$r$ projection $\hat{\Pi}^r_{\Vert{\cdot}\Vert_*\leq 1}[\cdot]$ does not change the produced sequences $\{\X_t\}_{t\geq 0}, \{\Y_t\}_{t\geq 0}$.

In particular, the following standard convergence rate guarantee \cite{Nesterov13} holds 
\begin{eqnarray*}
\forall t\geq 1:\quad f(\X_t) - f(\X^*) \leq \frac{\alpha+\beta}{2}\Vert{\X_0-\X^*}\Vert_F^2\left({1 - \sqrt{\alpha/\beta}}\right)^t.
\end{eqnarray*}
\end{theorem}

\begin{proof}
As in the proof of Theorem \ref{thm:PGD}, it suffices to show that for all $t\geq 0$, $\Vert{\Y_t-\X^*}\Vert_F \leq \delta(\X^*,r,1/\beta)$.
By the update rule in \eqref{eq:thm:AGD_SC:agd} we have that for all $t\geq 0$
\begin{eqnarray*}
\Vert{\Y_{t+1} - \X_{t+1}}\Vert_F \leq \frac{\sqrt{\beta}-\sqrt{\alpha}}{\sqrt{\beta}+\sqrt{\alpha}}\Vert{\X_{t+1}-\X_t}\Vert_F \leq \Vert{\X_{t+1}-\X_t}\Vert_F.
\end{eqnarray*}
Thus, we have that
\begin{eqnarray*}
\Vert{\Y_{t+1} - \X^*}\Vert_F &\leq & \Vert{\Y_{t+1} - \X_{t+1}}\Vert_F + \Vert{\X_{t+1} - \X^*}\Vert_F \\
&\leq &\Vert{\X_{t+1}-\X_t}\Vert_F + \Vert{\X_{t+1} - \X^*}\Vert_F \\
& \leq & \Vert{\X_{t}-\X^*}\Vert_F + 2\Vert{\X_{t+1} - \X^*}\Vert_F.  
\end{eqnarray*}
Thus, it suffices to show that for all $t\geq 0$, $\Vert{\X_t - \X^*}\Vert_F \leq \delta(\X^*,r,1/\beta)/3$. The proof is by simple induction. Clearly the claim holds for $t=0$. Suppose now that the claim holds for all iterations up to some $t\geq 0$. Then, it follows that the sequence $\{\X_i\}_{i\in[t+1]}$ produced by replacing the accurate projection with the rank-$r$ truncated projection $\hat{\Pi}^r_{\Vert{\cdot}\Vert_*\leq 1}[\cdot]$ is identical to one produced when using the accurate projection. Thus, by the convergence rate of the accelerated gradient method stated in the theorem, we have that
\begin{eqnarray*}
f(\X_{t+1}) - f(\X^*) &\leq &\frac{\alpha+\beta}{2}\Vert{\X_0-\X^*}\Vert_F^2\left({1 - \sqrt{\alpha/\beta}}\right)^{t+1} \\
&\leq& \frac{\alpha+\beta}{2}\Vert{\X_0-\X^*}\Vert_F^2.
\end{eqnarray*}
Now, using the strong-convexity of $f$ we have that
\begin{eqnarray*}
\Vert{\X_{t+1}-\X^*}\Vert_F^2 \leq \frac{2}{\alpha}\left({f(\X_{t+1})-f(\X^*)}\right) \leq \frac{\alpha+\beta}{\alpha}\Vert{\X_0-\X^*}\Vert_F^2.
\end{eqnarray*}
Thus, if $\Vert{\X_0-\X^*}\Vert_F \leq \frac{\sqrt{\alpha}}{3\sqrt{\alpha+\beta}}\delta(\X^*,r,1/\beta)$, it indeed holds that $\Vert{\X_{t+1}-\X^*}\Vert_F \leq \delta(\X^*,1/\beta,r)/3$ as needed.
\end{proof}

We remark that in the context of low-rank matrix optimization, especially in statisticaly-motivated settings (e.g., \cite{Agarwal10, Dropping16}), it is often assumed that the objective $f$ is not strongly-convex but only $(\alpha,k)$-\textit{restricted strongly convex}, which means in a nutshell that the standard strong convexity inequality
\begin{eqnarray*}
f(\X) - f(\Y) \leq (\X-\Y)\bullet\nabla{}f(\X) - \frac{\alpha}{2}\Vert{\X-\Y}\Vert_F^2
\end{eqnarray*}
holds for any two matrices $\X,\Y$ of rank at most $k$. It is not difficult to verify that for $k\geq r$, Theorems \ref{thm:PGD},\ref{thm:AGD_SC} imply local convergence with linear rate under the weaker assumption of restricted strong convexity of $f$. A full account of this argument is however beyond the scope of this paper.

We now turn to discuss the local convergence of the accelerated gradient method without strong-convexity (or restricted strong-convexity). Unfortunately,  unlike the case for the standard projected-gradient method or the strongly-convex case, the iterates of the method may in-principle step-outside of the ball around an optimal solution $\X^*$ in which the projected-gradient mapping is low-rank. For this reason, in the non-strongly convex case we require a stronger initialization condition which prohibits such divergence \footnote{In particular, we remark that the introduction of the function $R(\cdot)$ in the theorem is very similar in nature to use of the function $R_0(\cdot)$, which bounds the size of the initial level set, in the work of Nesterov on randomized coordinate descent methods \cite{Nesterov12}.}.  Also, for this result we assume the optimal solution is unique, i.e., $\mX^* = \{\X^*\}$. We note this assumption is quite mild since naturally the addition of a regularizing term of the form $\lambda\Vert{\X}\Vert_F^2$ to the objective with $\lambda >0$ being arbitrarily small, will result in such consequence.

\begin{theorem}\label{thm:AGD_notSC}[local convergence of Accelerated Gradient Method without strong convexity]
Assume there is a unique minimizer of $f$ over the unit trace-norm ball, i.e., $\mX^*=\{\X^*\}$.
Fix some integer $r$ such that $\min\{m,n\} \geq r \geq \#\sigma_1(\nabla{}f(\X^*))$.
%Fix some integer $r$ such that $\min\{m,n\} \geq r \geq \sup_{\X^*\in\mX^*}\#\sigma_1(\nabla{}f(\X^*))$.
Consider the function:
\begin{eqnarray*}
R(\X) := \sup_{\Vert{\Z}\Vert_*\leq 1: ~f(\Z) \leq f^* + c_0\beta\Vert\X-\X^*\Vert_F^2}\Vert{\Z-\X^*}\Vert_F,
%R(\X) := \sup_{\Vert{\Z}\Vert_*\leq 1: ~f(\Z) \leq f^* + c_0\beta\dist(\X,\mX^*)^2}\dist(\Z,\mX^*),
\end{eqnarray*}
%\begin{eqnarray*}
%R(\X) := \sup_{\Vert{\Z}\Vert_*\leq 1: ~f(\Z) \leq f(\X^*) + c_0\beta\Vert{\X-\X^*}\Vert_F^2}\Vert{\Z-\X^*}\Vert_F,
%\end{eqnarray*}
where $c_0$ is a universal constant to be specified in the sequel.
Consider Nesterov's Accelerated Gradient Method \cite{Nesterov13}, given by the update rule:
\begin{eqnarray}\label{eq:thm:AGD_notSC:agd}
%\forall t\geq 0: \qquad \X_{t+1}  & \gets &\Pi_{\Vert{\cdot}\Vert_*\leq 1}\left[{\Y_{t} - \frac{1}{\beta}\nabla{}f(\Y_{t})}\right] \\
\forall t\geq 0: \qquad \X_{t+1}  & \gets &\projgrad_{1/\beta}(\Y_t) \\
\Y_{t+1}  & \gets & \X_{t+1} + b_t(\X_{t+1} - \X_t), \nonumber 
\end{eqnarray}
where $\Y_0=\X_0\in\{\X ~|~\Vert{\X}\Vert_*\leq 1\}$, and $b_t = \frac{a_t(1-a_t)}{a_t^2+a_{t+1}}$, where $a_{t+1}\in(0,1)$ is a solution to the quadratic equation $a_{t+1}^2 = (1-a_{t+1})a_t^2$, and $a_0=1/2$.
Then, if $\max\{\Vert{\X_0-\X^*}\Vert_F,3R(\X_0)\} \leq \delta(\X^*,r,1/\beta)$
%$\min\{\dist(\X_0,\mX^*),3R(\X_0)\} \leq \delta(\mX^*,r,1/\beta)$
, we have that replacing the projection $\Pi_{\Vert\cdot\Vert_*\leq 1}[\cdot]$ in\eqref{eq:thm:AGD_notSC:agd} with the rank-$r$ projection $\hat{\Pi}^r_{\Vert{\cdot}\Vert_*\leq 1}[\cdot]$ does not change the produced sequences $\{\X_t\}_{t\geq 0}, \{\Y_t\}_{t\geq 0}$.

In particular, the following standard convergence rate guarantee holds 
\begin{eqnarray*}
\forall t\geq 1:\quad f(\X_t) - f(\X^*) \leq c_0\beta\Vert{\X_0-\X^*}\Vert_F^2/t^2,
\end{eqnarray*}
where $c_0$ is the suitable universal constant.
\end{theorem}

\begin{proof}
We prove by induction that for all $t\geq 0$, %$\dist(\Y_t,\mX^*) \leq \delta = \delta(\mX^*,1/\beta,r)$. 
$\Vert{\Y_t-\X^*}\Vert_F \leq \delta = \delta(\X^*,1/\beta,r)$.
Clearly the induction holds for the base case $t=0$, by our choice of $\Y_0$.
Suppose now the induction holds for all iterations up to (and including) iteration $t$. Then, it follows that the sequence $\{\X_i\}_{i\in[t+1]}$ produced by replacing the accurate projection with the rank-$r$ truncated projection $\hat{\Pi}^r_{\Vert{\cdot}\Vert_*\leq 1}[\cdot]$ is identical to one produced when using the accurate projection. Thus, %denoting by $\X_0^*$ the projection of $\X_0$ onto the optimal set $\mX^*$, we have 
by the convergence rate of the accelerated gradient method stated in the theorem, we have that 
\begin{eqnarray*}
f(\X_{t+1}) - f(\X^*) \leq  c_0\beta\Vert{\X_0-\X^*}\Vert_F^2/(t+1)^2,
%f(\X_{t+1}) - f(\X_0) \leq  c_0\beta\Vert{\X_0-\X_0^*}\Vert_F^2/(t+1)^2,
\end{eqnarray*}
which implies by the definition of the function $R(\cdot)$ that
\begin{eqnarray*}
\Vert\X_{t+1}-\X^*\Vert_F \leq R(\X_0) \leq \delta/3.
%\dist(\X_{t+1},\mX^*) \leq R(\X_0) \leq \delta/3.
\end{eqnarray*}

%Thus, by the definition of the function $R(\cdot)$, it follows that $\dist(\X_{t+1},\mX^*) \leq R(\X_0)$. Thus, by the assumption on the choice of $\X_0$, we have that $\dist(\X_{t+1},\mX^*) \leq \delta/3$.

A simple calculation shows that $b_t\in[0,1]$ and thus we have that 
\begin{eqnarray*}
\Vert{\Y_{t+1}-\X^*}\Vert_F &=&  \Vert{\X_{t+1}-\X^* + b_t(\X_{t+1}-\X_t)}\Vert_F\\
&\leq & \Vert{\X_{t+1}-\X^*}\Vert_F + b_t\Vert{\X_{t+1}-\X_t}\Vert_F \\
&\leq & (1+b_t)\Vert{\X_{t+1}-\X^*}\Vert_F + b_t\Vert{\X_t-\X^*}\Vert_F \\
&\leq & 2\Vert{\X_{t+1}-\X^*}\Vert_F + \Vert{\X_t-\X^*}\Vert_F \leq \delta,
%\dist(\Y_{t+1},\mX^*) &\leq& \Vert{\Y_{t+1}-\X_{t+1}^*}\Vert_F = \Vert{\X_{t+1}-\X_{t+1}^* + b_t(\X_{t+}-\X_t)}\Vert_F\\
%&\leq & \Vert{\X_{t+1}-\X_{t+1}^*}\Vert_F + b_t\Vert{\X_{t+1}-\X_t}\Vert_F
\end{eqnarray*}
hence the result follows.

%The proof is by induction. By the choice of $\X_0,\Y_0$ it clearly holds that

%Similarly to the proof of Theorem \ref{thm:AGD_SC}, it suffices to show that for all $t\geq 0$, $\dist(\Y_t,\mX^*) \leq \delta(\mX^*,r,1/\beta)$, and then the result follows directly from Lemma \ref{lem:goodProj}.

%From simple calculation, one can see that for all $t$, $b_t\in[0,1]$. For all $t\geq 0$ let us denote by $\X_t^*$ the projection of $\X_t$ onto the optimal set $\mX^*$.

 %Thus, as in the proof of Theorem \ref{thm:AGD_SC}, from Eq. \eqref{eq:thm:AGD_SC:agd}, we have that for all $t$ that
%\begin{eqnarray*}
%\Vert{\Y_{t+1}-\X_{t+1}}\Vert_F \leq \Vert{\X_{t+1} - \X_t}\Vert_F,
%\end{eqnarray*}
%which, as in the proof of Theorem \ref{thm:AGD_SC}, implies that
%\begin{eqnarray*}
%\Vert{\Y_{t+1}-\X^*}\Vert_F \leq 2\Vert{\X_{t+1} - \X^*}\Vert_F + \Vert{\X_t-\X^*}\Vert_F.
%\end{eqnarray*}
%Thus, it remains to show that for all $t\geq 1$, $\Vert{\X_t-\X^*}\Vert_F \leq \delta(r,1/\beta)/3$.

%By the convergence guarantee in the theorem, we have that for all $t\geq 1$, $f(\X_t) - f(\X^*) \leq c_0\beta\Vert{\X_t-\X_0}\Vert_F^2$, which by the definition of $R(\X_0)$ and the assumption of the lemma, implies that $\Vert{\X_t-\X_0}\Vert_F \leq R(\X_0) \leq \delta(r,1/\beta)/3$, as needed.

\end{proof}

We remark that while our focus with respect to accelerated gradient methods was specifically on Nesterov's method \cite{Nesterov13}, similar results can be obtained in a similar manner to other accelerated methods such as FISTA \cite{FISTA} (which we use for our experiments).

\subsection{Initialization, Certificates for Convergence, and Applications to Global Optimization}

All convergence results described above are local and hold only from a ``warm-start" initialization which naturally depends on the lower bound on the radius $\delta(\X^*,\eta,r)$, which in turn crucially depends on the spectral gap in $\nabla{}f(\X^*)$, a quantity which is clearly not available in general, or easily estimated. Also, all methods require an upper bound estimate $r$ on the multiplicity of the largest singular value - $\#\sigma_1(\nabla{}f(\X^*))$.
While the focus of this current work is mainly theoretical - attempting to provide theory in support of the empirical success of these methods (which is also reported in Section \ref{sec:empirics}), we now discuss several aspects of more practical flavor concerning the use of the above results.
 
First, we note that if the objective is strongly convex with parameter $\alpha>0$, then clearly a standard bound on the approximation error with respect to the function value could be translated to a bound on the distance to $\X^*$. This could be used to analyze the complexity of a scheme that applies basically any globally-convergent optimization method at first, and then switches to the methods described here, once close enough to $\X^*$.
%This reasoning could also be used when only $(\alphr,r)$-restricted strong convexity holds, i.e., strong convexity only w.r.t. matrices of rank at most $r$ (see short discussion above). In this case, given a feasible point $\X$ and denoting $\X_r$ the closest feasible rank-$r$ matrix, and assuming that 

More importantly, we would like to highlight that from a practical point of view, it does not matter whether the method is initialized as described in the above results, and knowledge of the spectral gap parameter is also not relevant in a sense. In practice, it is mainly important that the truncated-SVD-based projection is indeed the correct Euclidean projection, which facilitates the correct convergence of the methods discussed. This however could be easily verified on each iteration, as already mentioned in Lemma \ref{lem:traceNormProj}, and as we now reemphasize. The moderate price for obtaining such a certificate for the projection, is to compute a rank-$(r+1)$ SVD instead of only a rank-$r$ SVD on each iteration.

\begin{observation}[follows directly from Lemma \ref{lem:traceNormProj}]\label{obsrv:verify}
Let $\X\in\reals^{m\times n}$ and consider the point $\Y = \X-\eta\nabla{}f(\X)$, for some step-size $\eta\geq 0$. Then, using the top $(r+1)$ components of the SVD of $\Y$, it is possible to both compute $\widehat{\Pi}^r_{\Vert{\cdot}\Vert_*\leq 1}[\Y]$ and to certify whether  $\widehat{\Pi}_{\Vert{\cdot}\Vert_*\leq 1}[\Y] = \Pi_{\Vert{\cdot}\Vert_*\leq 1}[\Y]$ holds or not.
\end{observation}
Thus, Observation \ref{obsrv:verify} gives us a way to verify that the low-rank-projection-based method indeed converges correctly, without requirement of unavailable parameters such as the spectral gap in  $\nabla{}f(\X^*)$.

We can take the above observation a step further. Observation \ref{obsrv:verify} allows us for instance to combine a low-rank projected gradient method, such as in Theorem \ref{thm:PGD}, with an efficient globally-convergent method, such as the conditional gradient method (aka Frank-Wolfe method, see for instance \cite{Jaggi10}) with line-search, to obtain an improved hybrid globally-convergent method. This could again be done without tuning of parameters such as the spectral gap. Once the conditional gradient method gets close enough to an optimal solution, it switches to the low-rank projected gradient method.  Concretely, consider updating the feasible iterate on each step $t$, $\X_t$ according to the following scheme.

{\center{
  \fbox{
    \begin{tabular}{@{}l@{}}
    Given the feasible iterate $\X_t$ at step $t$ of the algorithm:\\
    \\
     $\circ$ Compute the top $r+1$ SVD components of $\Y_t = \X_t - \beta^{-1}\nabla{}f(\X_t)$. \\
     \\
     $\circ$ If $\widehat{\Pi}^r_{\Vert{\cdot}\Vert_*\leq 1}[\Y_t] = \Pi_{\Vert{\cdot}\Vert_*\leq 1}[\Y_t]$ (using Observation \ref{obsrv:verify}) then \\
     \\
     $\quad \X_{t+1} \gets \widehat{\Pi}^r_{\Vert{\cdot}\Vert_*\leq 1}[\Y_t]$. (proj. grad. update)\\
     \\
     $\circ$ Otherwise, let $(\u,\v)\in\reals^m\times\reals^n$ be the top singular vectors of $\nabla{}f(\X_t)$. \\
     \\
     $\quad \X_{t+1} \gets (1-\eta_t)\X_t + \eta_t(-\u\v^{\top})$, where\\
     \\
     $\quad \eta_t\gets\arg\min_{\eta\in[0,1]}f\left({ (1-\eta)\X_t + \eta(-\u\v^{\top})}\right)$. (cond. grad. update)
    \end{tabular}
}}}
\\

It can be seen that in worst case, the above scheme requires one rank-$(r+1)$ SVD and an additional one rank-one SVD. Moreover, since both the projected gradient method and the conditional gradient method with line-search are descent methods, it is not difficult to show that the above scheme converges globally (i.e., without any warm-start requirement) with rate $O(\beta/t)$. Moreover, as can be seen, once the above hybrid method gets into the ball of radius $\delta(\X^*,1/\beta,r)$ around $\X^*$, only low-rank projected gradient updates are used, and hence from that point onwards, the method only maintains iterates of rank at most $r$. This improves the storage issue with the standard conditional gradient method, that typically maintains iterates with rank that grows linearly with the iteration counter $t$, and thus after many iterations requires to store high-rank matrices in memory which is inefficient for large-scale problems. For a formal and full account of such an argument we refer the reader to the sequel work \cite{Garber19FW} (see Section 2.3).

We note that while all methods considered (including the above hybrid approach) require to set the SVD rank parameter $r$, this task should not be very difficult in practice since $r$ is only required to satisfy $r \geq \#\sigma_1(\nabla{}f(\X^*))$, and so, any upper-bound will work (in particular, as indicated in Theorem \ref{thm:goodProj}, increasing $r$ will also increase the radius $\delta(\X^*,\eta,r)$). Moreover, again with the use of Observation \ref{obsrv:verify}, it could be easily verified during the run if the parameter $r$ is set correctly. The issue of setting a rank parameter $r$ is inherent also in non-convex methods for low-rank optimization which consider an explicit factorization of the matrix variable (see the related work section).

Finally, we note that another practical issue when using low-rank SVD computations is the fact that in practice these computations are never accurate and are only approximated using fast Krylov subspace methods such as power iterations or the Lanczos algorithm (see for instance \cite{musco2015randomized}). However, often in practice, setting the accuracy of these methods when applied with an iterative gradient method is not difficult. Accounting for such approximation errors in the convergence analysis of first-order methods, while somewhat technical, is often straightforward and has been discussed in many recent works (see for instance \cite{Jaggi13b, Garber16a, allen2017linear, soltani2017fast}). Since these considerations are orthogonal to the main arguments introduced in this work, we omit such technical details, and assume all SVD computations are accurate.

\subsection{Stochastic Gradient Descent with mini-batches}

While our focus in this current paper is on deterministic first-order methods, we briefly outline application of our results to the recently highly-popular Stochastic Gradient Descent (SGD) method (see for instance \cite{Bubeck15}).

\begin{assumption}\label{ass:sgd}
Suppose that $f:\reals^{m\times n}\rightarrow\reals$ is convex, $\beta$-smooth and $\nabla{}f$ is nonzero on the unit trace-norm ball. Assume further that $f$ is given by a stochastic sampling oracle which, when queried with a matrix $\X$, returns a random matrix $\hat{\G}$ such that
\begin{enumerate}
\item
$\E[\hat{\G}|\X] = \nabla{}f(\X)$,
\item
$\Vert{\hat{\G}}\Vert_F \leq G$ for some scalar $G\geq 0$.
%$\E[\Vert{\hat{\G}-\nabla{}f(\X)}\Vert_F^2~|~\X] \leq \ni^2$, for some $\ni \geq 0$.
\end{enumerate}
\end{assumption}

\begin{lemma}\label{lem:SGDrank}
Let $f:\reals^{m\times n}:\rightarrow\reals$ be a function that satisfies Assumption \ref{ass:sgd}. Fix some $\X^*\in\mX^*$, a step-size $\eta >0$ and an integer $r'\in\{r=\#\sigma_1(\nabla{}f(\X^*)),\dots\min\{m,n\}-1\}$. Let $\X\in\reals^{m\times n}$ be such that $\Vert{\X-\X^*}\Vert_F \leq \frac{\eta(\mu_1-\mu_{r'+1})}{4(1+\eta\beta)}$, where $\mu_i,i=1,\dots,\min\{m,n\}$ are the singular values of $\nabla{}f(\X^*)$. Let $\hat{\G}_1,\dots,\hat{\G}_k$ be stochastic gradients of $f$, produced by $k$ queries with the matrix $\X$ to the stochastic oracle of $f$,
and consider the projected stochastic gradient mapping: 
\begin{eqnarray*}
\widehat{\projgrad}_{\eta,k}(\X) := \Pi_{\Vert\cdot\Vert_*\leq 1}\left[{\X-\frac{\eta}{k}\sum_{i=1}^k\hat{\G}_i}\right].
\end{eqnarray*}
Then, if 
\begin{eqnarray*}
k \geq \frac{128G^2}{3(\mu_1-\mu_{r'+1})^2}\log\left({(m+n)\min\{m,n\}}\right),
\end{eqnarray*}
it holds that
\begin{eqnarray*}
\E\left[{\rank\left({\widehat{\projgrad}_{\eta,k}(\X)}\right)}\right] \leq r'+1.
\end{eqnarray*}
\end{lemma}

\begin{proof}
Clearly,
\begin{eqnarray*}
\E\left[{\rank\left({\widehat{\projgrad}_{\eta,k}(\X)}\right)}\right] &\leq &r'\cdot\Pr\left({\rank\left({\widehat{\projgrad}_{\eta,k}(\X)}\right)\leq r'}\right) \\
&&+ \min\{m,n\}\cdot\left({1 - \Pr\left({\rank\left({\widehat{\projgrad}_{\eta,k}(\X)}\right)\leq r'}\right)}\right).
\end{eqnarray*}

Let us denote $\Y^* = \X^*-\eta\nabla{}f(\X^*)$, $\Y = \X-\eta\nabla{}f(\X)$, and $\hat{\Y} = \X - \eta\hat{\nabla}f(\X)$, where we use the notation $\hat{\nabla}f(\X) = \frac{1}{k}\sum_{i=1}^k\hat{\G_i}$. Let us further denote by $\gamma_1,\dots,\gamma_{\min\{m,n\}}$ the singular values of $\Y$ and by $\hat{\gamma}_1,\dots,\hat{\gamma}_{\min\{m,n\}}$ the singular values of $\hat{\Y}$. 

Going through the steps of the proof of Theorem \ref{thm:goodProj} we can see that a sufficient condition for $\rank\left({\widehat{\projgrad}_{\eta,k}(\X)}\right) \leq r'$ to hold is that
\begin{eqnarray}\label{eq:sgd:1}
\sum_{i=1}^{r} \hat{\gamma}_i - r\hat{\gamma}_{r'+1} \geq 1.
\end{eqnarray}

Using Weyl's inequality for the singular values we have that
\begin{eqnarray*}
\sum_{i=1}^{r} \hat{\gamma}_i - r\hat{\gamma}_{r'+1} &\geq&  \sum_{i=1}^{r} \gamma_i - r\gamma_{r'+1} - 2r\sigma_1(\Y-\hat{\Y}) \\
&\geq & 1 + \eta{}r(\mu_1-\mu_{r'+1}) - (r+\sqrt{r})(1+\eta\beta)\Vert{\X-\X^*}\Vert_F\\
&& - 2r\Vert{\Y-\hat{\Y}}\Vert,
\end{eqnarray*}
where the last inequality follows from the lower bound on $ \sum_{i=1}^{r} \gamma_i - r\gamma_{r'+1}$ developed in the proof of Theorem \ref{thm:goodProj}.

%Similarly to the bounds developed in in the proof of Theorem \ref{thm:goodProj}
%\begin{eqnarray}\label{eq:sgd:1}
%\Vert{\hat{\Y}-\Y^*}\Vert \leq \frac{\eta}{2}(\mu_1-\mu_{r+1}),
%\end{eqnarray}
%where $\mu_i,i=1,\dots,\min\{m,n\}$ are the singular values of $\nabla{}f(\X^*)$.

It holds that
\begin{eqnarray*}
\Vert{\hat{\Y}-\Y}\Vert = \Vert{\X-\eta\hat{\nabla}f(\X)-\X+\eta\nabla{}f(\X)}\Vert = \eta\Vert{\hat{\nabla}{}f(\X)-\nabla{}f(\X)}\Vert.
\end{eqnarray*}

%\begin{eqnarray*}
%\Vert{\hat{\Y}-\Y^*}\Vert &=& \Vert{\X-\eta\hat{\nabla}f(\X)-\X^*+\eta\nabla{}f(\X^*)}\Vert \\
%&\leq & \Vert{\X-\X^*}\Vert_F + \eta\Vert{\hat{\nabla}f(\X)-\nabla{}f(\X^*)}\Vert \\
%&\leq &\Vert{\X-\X^*}\Vert_F + \eta\Vert{\nabla{}f(\X)-\nabla{}f(\X^*)}\Vert_F + \Vert{\hat{\nabla}{}f(\X)-\nabla{}f(\X)}\Vert \\
%&\leq &(1+\eta\beta)\Vert{\X-\X^*}\Vert_F + \Vert{\hat{\nabla}{}f(\X)-\nabla{}f(\X)}\Vert.
%\end{eqnarray*}

Thus, for \eqref{eq:sgd:1} to hold, it suffices that
\begin{eqnarray*}
(r+\sqrt{r})(1+\eta\beta)\Vert{\X-\X^*}\Vert_F + 2\eta{}r\Vert{\hat{\nabla}{}f(\X)-\nabla{}f(\X)}\Vert \leq \eta{}r(\mu_1-\mu_{r'+1}).
\end{eqnarray*}

Under our assumption on $\Vert{\X-\X^*}\Vert_F$, a sufficient condition for \eqref{eq:sgd:1} to hold is that
\begin{eqnarray*}
\Vert{\hat{\nabla}{}f(\X)-\nabla{}f(\X)}\Vert \leq \frac{1}{4}(\mu_1-\mu_{r'+1}).
%\Vert{\hat{\nabla}{}f(\X)-\nabla{}f(\X)}\Vert &\leq& \frac{\eta}{2}(\mu_1-\mu_{r+1}) - (1+\eta\beta)\Vert{\X-\X^*}\Vert_F \nonumber \\
%&\leq &\frac{\eta}{4}(\mu_1-\mu_{r+1}),
\end{eqnarray*}

%where the last inequality follows from the assumption $\Vert{\X-\X^*}\Vert_F$.
Using the Matrix-Bernstein inequality \cite{Tropp12}, we have that
\begin{eqnarray*}
\Pr\left({\Big\|\hat{\nabla}f(\X)-\nabla{}f(\X)\Big\| > \frac{1}{4}(\mu_1-\mu_{r'+1})}\right) \leq (m+n)\cdot\exp\left({-\frac{3(\mu_1-\mu_{r'+1})^2k}{128G^2}}\right).
\end{eqnarray*}
Thus, for 
\begin{eqnarray*}
k \geq \frac{128G^2}{3(\mu_1-\mu_{r'+1})^2}\log\left({(m+n)\min\{m,n\}}\right),
\end{eqnarray*}
we indeed obtain
\begin{eqnarray*}
\E\left[{\rank\left({\widehat{\projgrad}_{\eta,k}(\X)}\right)}\right] \leq r'+1.
\end{eqnarray*}
\end{proof}

Using the same concentration arguments one can easily obtain a version of Lemma \ref{lem:SGDrank} that holds with high probability, however these are beyond the scope of this current paper.

\section{Trace-Norm Regularization}\label{sec:traceReg}

In this section we focus our attention to a different optimization problem closely related to Problem \eqref{eq:optProb}, which considers the trace-norm as a regularizer instead of a feasible constraint.

\begin{eqnarray}\label{eq:regProb}
\min_{\X\in\reals^{m\times n}}f(\X) + \Vert{\X}\Vert_*.
\end{eqnarray}

In this setting, first-order methods for composite optimization which handle the trace-norm in a close-form manner, e.g., FISTA \cite{FISTA}, rely on computing the soft-thresholding operator:
\begin{eqnarray}
\mT_{\eta}\left({\X-\eta\nabla{}f(\X)}\right)  = \sum_{i=1}^{\min\{m,n\}}\max\{\sigma_i-\eta,0\}\u_i\v_i^{\top},
\end{eqnarray}
where $\sum_{i=1}^{\min\{m,n\}}\sigma_i\u_i\v_i^{\top}$ is the SVD of $\Y := \X-\eta\nabla{}f(\X)$, see for instance \cite{Cai10}.

In this section we overload notation and define the corresponding proximal-gradient mapping:
\begin{eqnarray*}
\projgrad_{\eta}(\X) := \mT_{\eta}(\X-\eta\nabla{}f(\X)).
\end{eqnarray*}
Similarly, for any optimal solution $\X^*\in\mX^*$, step-size $\eta$ and $r\in\{1,\dots,\min\{m,n\}\}$,
we also overload the notation $\delta(\X^*,\eta,r)$ and define it as
\begin{eqnarray*}
\sup\delta\geq 0 \quad \textrm{s.t.} \quad \forall\X\in\ball(\X^*,\delta):\rank\left({\projgrad_{\eta}(\X)}\right)\leq r,
\end{eqnarray*}
where, as before, $\ball(\X,R)$ denotes the Euclidean ball of radius $R$ centered at $\X$.

The following theorem is analogues to Theorem \ref{thm:goodProj}.

\begin{theorem}\label{thm:goodProjReg}
Let $f:\reals^{m\times n}\rightarrow\reals$ be $\beta$-smooth and convex. Assume $\nabla{}f$ is non-zero over $\mX^*$ and fix some $\X^*\in\mX^*$. Let $r$ denote the multiplicity of $\sigma_1(\nabla{}f(\X^*))$, and let $\mu_1,\mu_2,\dots,\mu_{\min\{m,n\}}$ denote the singular values of $\nabla{}f(\X^*)$ (including multiplicities). Then, for any $\eta >0$ 
it holds that
\begin{eqnarray}\label{eq:goodProjReg1}
\delta(\X^*,\eta,r) \geq  \frac{\eta{}(\mu_1-\mu_{r+1})}{1+\eta\beta}.
\end{eqnarray}
More generally, for any $r'\in\{r,\dots,\min\{m,n\}-1\}$ and any $\eta >0$, it holds that 
\begin{eqnarray}\label{eq:goodProjReg2}
\delta(\X^*,\eta,r') \geq  \frac{\eta(\mu_1-\mu_{r'+1})}{1+\eta\beta}.
\end{eqnarray}
Moreover, for any $r'\in\{r,\dots,\min\{m,n\}-r\}$ and any $\eta >0$, it holds that
\begin{eqnarray}\label{eq:goodProjReg3}
\delta(\X^*,\eta,r'+r-1) \geq \frac{\sqrt{r}\eta(\mu_1-\mu_{r'+1})}{1+\eta\beta}.
\end{eqnarray}
\end{theorem}

\begin{proof}
Since $\X^*$ is an optimal solution it follows that 
\begin{eqnarray*}
\projgrad_{\eta}(\X^*) = \mT_{\eta}(\X^*-\eta\nabla{}f(\X^*)) = \X^*. 
\end{eqnarray*}

We first note that Lemma \ref{lem:eigsOfOptGrad} also applies for the regularized problem, Problem \eqref{eq:regProb} (extending its proof to the regularized case is straightforward). 
Thus, we can write the SVD of $\X^*$ as $\X^* = \sum_{i=1}^{r^*}\sigma_i\u_i\v_i$, where each pair $(-\u_i,\v_i)$ is a singular vectors pair corresponding to the top singular value of $\nabla{}f(\X^*)$ - $\sigma_1(\nabla{}f(\X^*)$. Thus, we have that
\begin{eqnarray*}
\X^* &=& \sum_{i=1}^{r^*}\sigma_i\u_i\v_i^{\top} = \mT_{\eta}\left({\sum_{i=1}^{r^*}\sigma_i\u_i\v_i^{\top} - \eta\sum_{i=1}^{\min\{m,n\}}\mu_i(-\u_i\v_i^{\top})}\right) \\
&=& \mT_{\eta}\left({\sum_{i=1}^{\min\{m,n\}}(\sigma_i+\eta\mu_i)\u_i\v_i^{\top}}\right) = \sum_{i=1}^{\min\{m,n\}}\max\{\sigma_i+\eta\mu_i-\eta,0\}\u_i\v_i^{\top}. 
\end{eqnarray*}

Thus, it follows that
\begin{eqnarray*}
\forall i\in\{1,\dots,r^*\}: ~ \sigma_i = \sigma_i + \eta\mu_i-\eta \quad &\Longrightarrow& \quad \forall i\in\{1,\dots,r^*\}:~\mu_i = 1 \\
&\Longrightarrow& \quad \forall i\in\{1,\dots,r\}:~\mu_i = 1.
\end{eqnarray*}

In particular, it follows that for all $i>r$, $\mu_i < 1$.

Now, given some point $\X$, let us denote $\Y := \X-\eta\nabla{}f(\X)$, and let us denote the singular values of $\Y$ by $\gamma_1,\dots,\gamma_{\min\{m,n\}}$.

Using the notation $\Y^* := \X^* - \eta\nabla{}f(\X^*)$, it holds via Weyl's inequality that for all $r' \geq r$
\begin{eqnarray*}
\gamma_{r'+1} \leq  \sigma_{r'+1}(\Y^*) + \sigma_1(\Y-\Y^*) \leq \eta\mu_{r'+1} + (1+\eta\beta)\Vert{\X-\X^*}\Vert_F,
\end{eqnarray*}
where the bound on $\sigma_1(\Y-\Y^*)$ follows as in the proof of Theorem \ref{thm:goodProj}.

Clearly, by definition of $\mT_{\eta}(\cdot)$ it follows that if $\gamma_{r'+1} < \eta$ then $\rank\left({\mT_{\eta}(\Y)}\right) \leq r'$.

Thus, it follows that if $\Vert{\X-\X^*}\Vert_F \leq \frac{\eta(1-\mu_{r'+1})}{1+\eta\beta} = \frac{\eta(\mu_1-\mu_{r'+1})}{1+\eta\beta}$, then $\mT_{\eta}\left({\X-\eta\nabla{}f(\X)}\right)$ has rank at most $r'$. This proves \eqref{eq:goodProjReg1}, \eqref{eq:goodProjReg2}.

Alternatively, for any $r'' \geq r'+r-1$, using Weyl's inequality, we have that
\begin{eqnarray*}
\gamma_{r''+1} &\leq & \sigma_{r'+1}(\Y^*) + \sigma_{r''-r'+1}(\Y-\Y^*) = \eta\mu_{r'+1} + \sqrt{\sigma^2_{r''-r'+1}(\Y-\Y^*)} \\
&\leq &\eta\mu_{r'+1} + \sqrt{\frac{1}{r''-r'+1}\Vert{\Y-\Y^*}\Vert_F^2} = \eta\mu_{r'+1} + \frac{1}{\sqrt{r''-r'+1}}\Vert{\Y-\Y^*}\Vert_F \\
&\leq &\eta\mu_{r'+1} + \frac{1+\eta\beta}{\sqrt{r''-r'+1}}\Vert{\X-\X^*}\Vert_F
\end{eqnarray*}

Thus, if $\Vert{\X-\X^*}\Vert_F \leq \frac{\eta\sqrt{r''-r'+1}(\mu_1-\mu_{r'+1})}{1+\eta\beta}$ it follows that $\mT_{\eta}\left({\X-\eta\nabla{}f(\X)}\right)$ has rank at most $r''$. This proves \eqref{eq:goodProjReg3} when taking $r''=r'+r-1$.
\end{proof}

\section{Optimization with Positive Semidefinite Matrices}\label{sec:SDP}

In this section we consider the related problem of optimization over positive semidefinite matrices with unit trace. Towards this end we define the spectrahedron: $\mS_n := \{\X\in\mbS^nn ~ |~ \X\succeq 0,~\trace(\X) = 1\}$, where $\mbS^n$ denotes the space of $n\times{}n$ real symmetric matrices, and consider the following optimization problem:
\begin{eqnarray*}
\min_{\X\in\mS_n}f(\X),
\end{eqnarray*}
where, as before, $f:\reals^{n\times n}\rightarrow\reals$ is assumed to be convex and $\beta$-smooth. For ease of presentation, throughout this section we assume that the gradient vector is always a symmetric matrix, i.e., $\nabla{}f(\X) \in\mbS^n$ for all $\X\in\mbS^n$.

In this section we again overload notation and consider the projected-gradient mapping w.r.t. $\mS_n$:
\begin{eqnarray*}
\projgrad_{\eta}(\X) := \Pi_{\mS_n}[\X-\eta\nabla{}f(\X)].
\end{eqnarray*}

Similarly, for any optimal solution $\X^*\in\mX^*$, step-size $\eta$ and integer $r\in\{1,\dots,n\}$,
we also overload the notation $\delta(\X^*,\eta,r)$ and define it as
\begin{eqnarray*}
\sup\delta\geq 0 \quad \textrm{s.t.} \quad \forall\X\in\ball(\X^*,\delta):\rank\left({\projgrad_{\eta}(\X)}\right)\leq r.
\end{eqnarray*}

We begin by recalling the structure of the Euclidean projection onto $\mS_n$.

\begin{lemma}[projection onto the spectrahedron]\label{lem:SDPProj}
Let $\X\in\sym^n$ and consider its eigendecomposition $\X = \sum_{i=1}^n\lambda_i\v_i\v_i^{\top}$, where $\lambda_1 \geq \lambda_2 \dots\geq \lambda_n$. The Euclidean projection of $\X$ onto $\mS_n$ is given by
\begin{eqnarray*}
\Pi_{\mS_n}[\X] = \sum_{i=1}^{n}\max\{0,\lambda_i-\lambda\}\v_i\v_i^{\top},
\end{eqnarray*}
where $\lambda\in\reals$ is the unique solution to the equation $\sum_{i=1}^{n}\max\{0,\lambda_i-\lambda\} = 1$.

Moreover, there exists $r\in\{1,\dots,n-1\}$ such that $\sum_{i=1}^r\lambda_i \geq 1+r\lambda_{r+1}$ if and only if $\rank\left({\Pi_{\mS}[\X]}\right) \leq r$. 
\end{lemma}

\begin{proof}
The first part of the Lemma is a well-known fact, see for instance \cite{Beck17}. 

To see why the second part holds, we first observe that in case $\sum_{i=1}^r\lambda_i \geq 1+r\lambda_{r+1}$, it must hold that $\lambda \geq \lambda_{r+1}$. This is true since if $\lambda < \lambda_{r+1}$ we have that
\begin{eqnarray*}
\trace\left({\Pi_{\mS_n}[\X]}\right) &=& \sum_{i=1}^{n}\max\{0,\lambda_i-\lambda\} \geq \sum_{i=1}^{r}\max\{0,\lambda_i-\lambda\} \\
&= &\sum_{i=1}^{r}(\lambda_i-\lambda) > (1+r\lambda_{r+1}) - r\lambda_{r+1} = 1.
\end{eqnarray*}
Thus, it follows that $\lambda \geq \lambda_{r+1}$. However, then it clearly follows that for all $i\geq r+1$, $\lambda_i-\lambda \leq 0$ which implies that $\rank\left({\Pi_{\mS_n}[\X]}\right) \leq r$. The reversed direction follows from the same reasoning.
\end{proof}

The following lemma is analogues to Lemma \ref{lem:eigsOfOptGrad} and its proof (which we omit) follows the same reasoning.

\begin{lemma}\label{lem:eigsOfOptGradSDP}
Let $\X^*\in\mX^*$ be any optimal solution and write its eigendecomposition as $\X^* = \sum_{i=1}^r\lambda_i\v_i\v_i^{\top}$. Then, the gradient vector $\nabla{}f(\X^*)$ admits an eigendecomposition such that the set of vectors $\{\v_i\}_{i=1}^r$ is a set of top eigen-vectors of $(-\nabla{}f(\X^*))$ which corresponds to the eigenvalue $\lambda_1(-\nabla{}f(\X^*)) = -\lambda_n(\nabla{}f(\X^*))$.
\end{lemma}

The following theorem is analogues to Theorem \ref{thm:goodProj}.
\begin{theorem}\label{thm:goodProjSDP}
Let $f:\mbS^n\rightarrow\reals$ be $\beta$-smooth and convex. Assume $\nabla{}f$ is non-zero over $\mS_n$ and
fix some $\X^*\in\mX^*$. Let $r$ denote the multiplicity of $\lambda_n(\nabla{}f(\X^*))$, and let $\mu_1,\dots,\mu_n$ denote the eigenvalues of $\nabla{}f(\X^*)$ in non-increasing order. Then, for any $\eta >0$
it holds that
\begin{eqnarray}\label{eq:goodProjSDPres1}
\delta(\X^*,\eta,r) \geq  \frac{\eta(\mu_{n-r}-\mu_n)}{(1+1/\sqrt{r})(1+\eta\beta)}.
\end{eqnarray}
% and for any $\X\in\mbS_n$ (not necessarily feasible) satisfying $\Vert{\X-\X^*}\Vert_F \leq \delta(r,\eta) :=  \frac{\eta(\mu_{n-r}-\mu_n)}{(1+1/\sqrt{r})(1+\eta\beta)}$ it holds that $\rank\left({\Pi_{\mS}[\X-\eta\nabla{}f(\X)]}\right) \leq r$.

More generally, for  any $r'\in\{r,\dots,n-1\}$ and any $\eta >0$, it holds that
\begin{eqnarray}\label{eq:goodProjSDPres2}
\delta(r',\eta) \geq  \frac{\eta(\mu_{n-r'}-\mu_n)}{(1+1/\sqrt{r})(1+\eta\beta)}.
\end{eqnarray}
%for any $\X\in\mbS_n$ satisfying $\Vert{\X-\X^*}\Vert_F  \leq \delta(r',\eta) := \frac{\eta(\mu_{n-r'}-\mu_n)}{(1+1/\sqrt{r})(1+\eta\beta)}$ that  $\rank\left({\Pi_{\Vert\cdot\Vert_*\leq1}[\X-\eta\nabla{}f(\X)]}\right) \leq r'$.

Moreover, for any $r'\in\{r,\dots,n-r\}$ and any $\eta>0$, it holds that
\begin{eqnarray}\label{eq:goodProjSDPres3}
\delta(\X^*,\eta,r'+r-1) \geq \frac{\sqrt{r}\eta(\mu_{n-r'}-\mu_n)}{2(1+\eta\beta)}.
\end{eqnarray}
%for any $\X\in\mbS_n$ satisfying $\Vert{\X-\X^*}\Vert_F \leq \frac{\sqrt{r}\eta(\mu_{n-r'}-\mu_n)}{2(1+\eta\beta)}$ that $\rank\left({\Pi_{\Vert\cdot\Vert_*\leq1}[\X-\eta\nabla{}f(\X)]}\right) \leq r'' = r'+r-1$.
\end{theorem}

\begin{proof}
Fix a step-size $\eta>0$. Denote $\Y^* := \X^*-\eta\nabla{}f(\X^*)$ and let $\lambda_1,\dots,\lambda_n$ denote the eigenvalues of $\Y^*$ in non-increasing order. 

Let us denote by $r^* := \rank(\X^*)$.
It follows from Lemma \ref{lem:eigsOfOptGradSDP} that
\begin{eqnarray*}
\forall i\in[r^*]: \quad \lambda_i &=& \lambda_i(\X^*) - \eta\mu_n; \\
 \forall j>r^*: \quad \lambda_j &=& -\eta\mu_{n-j+1}.
\end{eqnarray*}

Since $\sum_{i=1}^{r^*}\lambda_i(\X^*) =\sum_{i=1}^{r}\lambda_i(\X^*) = \trace(\X^*)=1$, we have that
\begin{eqnarray}\label{eq:goodProjSDP1}
\xi(r') := \sum_{i=1}^r\lambda_i - r\cdot\lambda_{r'+1} - 1 
= 1-\eta{}r\mu_n  + r\eta\mu_{n-r'}-1 = \eta{}r(\mu_{n-r'}-\mu_n).
\end{eqnarray}

Now, given a matrix $\X$, denote $\Y := \X-\eta\nabla{}f(\X)$ and let $\gamma_1,\dots\gamma_n$ denote the eigenvalues of $\Y$ in non-increasing order. It holds that
\begin{eqnarray}\label{eq:goodProjSDP2}
\sum_{i=1}^{r}\gamma_i &\underset{(a)}{\geq} &\sum_{i=1}^{r}\lambda_i - \sum_{i=1}^{r}\lambda_i(\Y^*-\Y) \geq \sum_{i=1}^{r}\lambda_i -\sqrt{r\sum_{i=1}^{r}\lambda_i^2(\Y-\Y^*)} \nonumber \\
&\geq& \sum_{i=1}^{r}\lambda_i -\sqrt{r\sum_{i=1}^{n}\lambda_i^2(\Y-\Y^*)} = \sum_{i=1}^{r}\lambda_i  -\sqrt{r}\Vert{\Y-\Y^*}\Vert_F \nonumber\\
&=& \sum_{i=1}^{r}\lambda_i - \sqrt{r}\Vert{\X-\eta\nabla{}f(\X) - \X^* + \eta\nabla{}f(\X^*)}\Vert_F \nonumber\\
&\geq& \sum_{i=1}^{r}\lambda_i - \sqrt{r}\left({\Vert{\X-\X^*}\Vert_F + \eta\Vert{\nabla{}f(\X) - \nabla{}f(\X^*)}\Vert_F}\right) \nonumber\\
&\underset{(b)}{\geq}& \sum_{i=1}^{r}\lambda_i - \sqrt{r}(1+\eta\beta)\Vert{\X-\X^*}\Vert_F,
\end{eqnarray}
where (a) follows from Ky Fan's inequality for the eigenvalues, and (b) follows from the $\beta$-smoothness of $f$.

Also, similarly, it holds that
\begin{eqnarray}\label{eq:goodProjSDP3}
\gamma_{r'+1} \leq \lambda_{r'+1} + \lambda_1(\Y-\Y^*) \leq \lambda_{r'+1} + (1+\eta\beta)\Vert{\X-\X^*}\Vert_F. 
\end{eqnarray}

Combining Eq. \eqref{eq:goodProjSDP1}, \eqref{eq:goodProjSDP2}, \eqref{eq:goodProjSDP3}, we have that
\begin{eqnarray}\label{eq:goodProjSDP4}
\sum_{i=1}^{r'}\gamma_i - r'\gamma_{r'+1} &\geq & \sum_{i=1}^{r}\gamma_i - r\gamma_{r'+1} \nonumber \\
&\geq & \sum_{i=1}^{r}\lambda_i - \sqrt{r}(1+\eta\beta)\Vert{\X-\X^*}\Vert_F - r\left({\lambda_{r'+1} + (1+\eta\beta)\Vert{\X-\X^*}\Vert_F}\right) \nonumber\\
&=& \sum_{i=1}^{r}(\lambda_i-\lambda_{r'+1}) - (r+\sqrt{r})(1+\eta\beta)\Vert{\X-\X^*}\Vert_F\nonumber \\
&=& 1 + \xi(r') - (r+\sqrt{r})(1+\eta\beta)\Vert{\X-\X^*}\Vert_F.
\end{eqnarray}

Thus, it follows that if $\X$ satisfies:
\begin{eqnarray*}
\Vert{\X-\X^*}\Vert_F \leq  \frac{\xi(r')}{(r+\sqrt{r})(1+\eta\beta)} = \frac{\eta(\mu_{n-r'}-\mu_n)}{(1+1/\sqrt{r})(1+\eta\beta)},
\end{eqnarray*}
we have that $\sum_{i=1}^{r'}\gamma_i - r'\gamma_{r'+1} \geq 1$, which implies via Lemma \ref{lem:SDPProj}, that $\Pi_{\mS_n}\left[{\X-\eta\nabla{}f(\X)}\right]$ has rank at most $r'$. This proves \eqref{eq:goodProjSDPres1}, \eqref{eq:goodProjSDPres2}.

Alternatively, for any $r'' \geq r'+r-1$, using the more general version of Weyl's inequality, we can replace Eq. \eqref{eq:goodProjSDP3} with
\begin{eqnarray}\label{eq:goodProjSDP5}
\gamma_{r''+1} &\leq &\lambda_{r'+1} + \lambda_{r''-r'+1}(\Y-\Y^*) \leq \lambda_{r'+1} + \sqrt{\lambda_{r''-r'+1}^2(\Y-\Y^*)} \nonumber \\
&\leq & \lambda_{r'+1} + \sqrt{\frac{1}{r''-r'+1}\Vert{\Y-\Y^*}\Vert_F^2} \nonumber \\
&\leq & \lambda_{r'+1} + \frac{1}{\sqrt{r''-r'+1}}(1+\eta\beta)\Vert_F.
\end{eqnarray}

Thus, similarly to Eq. \eqref{eq:goodProjSDP4}, but replacing Eq. \eqref{eq:goodProjSDP3} with Eq. \eqref{eq:goodProjSDP5}, we obtain
\begin{eqnarray*}
\sum_{i=1}^{r''}\gamma_i - r''\gamma_{r''+1} &\geq &\sum_{i=1}^r\gamma_i - r\gamma_{r''+1} \geq \sum_{i=1}^r\lambda_i - \sqrt{r}(1+\eta\beta)\Vert{\X-\X^*}\Vert_F \\
&& - r\left({\lambda_{r'+1} + \frac{1}{\sqrt{r''-r'+1}}(1+\eta\beta)\Vert_F}\right) \\
&=& \sum_{i=1}^r(\lambda_i - \lambda_{r'+1}) - \left({\sqrt{r} + \frac{r}{\sqrt{r''-r'+1}}}\right)(1+\eta\beta)\Vert{\X-\X^*}\Vert_F \\
&=& 1+ \xi(r') - \left({\sqrt{r} + \frac{r}{\sqrt{r''-r'+1}}}\right)(1+\eta\beta)\Vert{\X-\X^*}\Vert_F.
\end{eqnarray*}

In particular, for $r'' = r' + r -1$, we obtain
\begin{eqnarray*}
\sum_{i=1}^{r''}\gamma_i - r''\gamma_{r''+1} \geq 1 + \xi(r') - 2\sqrt{r}(1+\eta\beta)\Vert{\X-\X^*}\Vert_F.
\end{eqnarray*}

Thus, it follows that if $\X$ satisfies:
\begin{eqnarray*}
\Vert{\X-\X^*}\Vert_F \leq \frac{\xi(r')}{2\sqrt{r}(1+\eta\beta)} = \frac{\sqrt{r}\eta(\mu_{n-r}-\mu_n)}{2(1+\eta\beta)},
\end{eqnarray*}
we have that $\Pi_{\mS_n}\left[{\X-\eta\nabla{}f(\X)}\right]$ has rank at most $r'+r-1$, which proves \eqref{eq:goodProjSDPres3}.
\end{proof}

\section{Motivating Empirical Evidence}\label{sec:empirics}
Our goal in this final section is to provide empirical evidence motivating our theoretical investigation. In particular, focusing on the well-studied low-rank matrix completion problem \cite{Candes09, Recht11, Jaggi10}, our results demonstrate that i) the optimal solution in real-world datasets is indeed low-rank, and ii) standard first-order methods, when initialized in a very simple and efficient way, indeed \textit{converge correctly} using only SVD computations with rank that either matches the rank of the optimal solution or exceeds it by a very small constant (1 or 2 in our experiments). To be clear, by the phrase \textit{converge correctly}, we mean that the methods produce exactly the same iterates as they would have produced when using a full-rank SVD computation on each iteration. This fact is verified on each iteration by checking that the condition stated at the end of  Lemma \ref{lem:traceNormProj} indeed holds.

To be more concrete, we consider the task of low-rank matrix completion in the following convex optimization formulation:

\begin{eqnarray*}
\min_{\Vert{\X}\Vert_*\leq\tau}\{f(\X):=\sum_{(i,j)\in{}S}(\X\bullet\matE_{i,j}-r_{i,j})^2\},
\end{eqnarray*}
where $S\subset[m]\times[n]$ is the set of observed entries and $r_{i,j}$ denotes the observed value.

We apply all algorithms with the following simple initialization scheme. We set the first iterate $\X_0$ to $\X_0 := \widehat{\Pi}^r_{\Vert{\cdot}\Vert_*\leq\tau}[\R]$, where
\begin{eqnarray*}
\R_{i,j} = \left\{ \begin{array}{ll}
        r_{i,j} & \mbox{if $(i,j)\in{}S$};\\
        \frac{1}{\vert{S}\vert}\sum_{(i,j)\in{}S}r_{i,j} & \mbox{if $(i,j)\notin{}S$}.\end{array} \right.
\end{eqnarray*}
In words: we construct a matrix which contains the observed matrix entries and every unobserved entry is set to the mean value of observed entries. We then compute the initialization by projecting the rank-$r$ truncated SVD of this matrix onto the trace-norm ball.

We use two highly popular datasets for the matrix completion problem, namely MovieLens100k ($943\times 1682$ matrix, 100,000 observed entries) and MovieLens1M ($6040\times 3952$ matrix, 1,000,209 observed entries)  \cite{Movielens16}. For each dataset we experiment with different  trace bounds (parameter $\tau$) which naturally influences the optimal value and the rank of the optimal solution.
For every dataset and trace bound $\tau$ we find the optimal mean-square-error ($\frac{1}{\vert{S}\vert}f(\X^*)$), the rank of the optimal solution, the multiplicity of the largest singular value in the gradient vector - $\#\sigma_1(\nabla{}f(\X^*))$, and the spectral gap between this largest singular value and the second largest (not counting multiplicities). These values are found by running any of the methods until a point with negligible dual-gap is reached, that is, we find a point $\X_{\epsilon}$ such that
\begin{eqnarray*}
\max_{\Vert{\V}\Vert_*\leq \tau}(\X_{\epsilon}-\V)\bullet\nabla{}f(\X_{\epsilon}) \leq \epsilon.
\end{eqnarray*}
Since $f$ is convex, this implies in particular that $f(\X_{\epsilon})-f(\X^*) \leq \epsilon$. For ML100k we use $\epsilon = 0.01$ and for ML1M we use $\epsilon = 0.5$.

For each of the algorithms tested - the standard projected-gradient method (PGD) \cite{Nesterov13} and FISTA \cite{FISTA}, we manually find the minimum rank parameter $r$ for which the algorithm converges correctly  from the "warm-start" initialization to the optimal solution using only a rank-$r$ SVD computation on each iteration. These parameters are recorded in the columns titled "PGD rank" and "FISTA rank" in Table \ref{table:empirical}. 

As can be seen in Table \ref{table:empirical}, the results clearly show that in all considered cases  it indeed holds that $\rank(\X^*) = \#\sigma_1(\nabla{}f(\X^*))$, and that $\rank(\X^*)$ is significantly smaller than $\min\{m,n\}$. We also see that PGD converges correctly to the optimal solution with rank that does not exceed that of the optimal solution, while in some cases FISTA requires rank a bit larger than that of the optimal solution (by at most 2). 

 \begin{table*}[h]\renewcommand{\arraystretch}{1.3}
{\footnotesize
\begin{center}
  \begin{tabular}{| c | c | c | c | c | c | c | c |}    \hline
    dataset &  trace &  $\rank(\X^*)$ &$\#\sigma_1(\nabla{}f^*)$ & FISTA rank & PGD rank & MSE & spectral $\textrm{gap}$\\\hline
\multirow{5}{3.6em}{ML100k} & $2500$ & $3$ & 3& $3$& 3 & 1.3589 & 5.5844 \\
& $3000$ & $10$  & $10$& $10$ & 10 & 0.9871 & 0.3234 \\
& $3500$ & $41$ & $41$& $42$ & 41& 0.7573 & 0.0456 \\
& $4000$ & $70$ & $70$& $71$ & 70& 0.5846 & 0.0227 \\
& $5000$ &  $117$ & $117$& $118$& 117 & 0.3314 & 0.0148 \\\hline    
\multirow{4}{3.6em}{ML1M}  & $10000$ & 3 & $3$& $3$& 3 & 1.2184 & 3.2861  \\ 
& $12000$ & 12 & $12$& $12$ & 12& 0.9043 & 1.2056 \\
 & $14000$ & $74$ & 74& $75$ & 74 & 0.7236 & 0.0698 \\
& $16000$ & 155 & 155& 157 & 155& 0.5918 & 0.0119 \\\hline    
%\multirow{1}{3.3em}{ML10M}  & $25000$ & 4 & $4$& $4$ & -& 1.5579 & 19.7865 \\ \hline
  \end{tabular}
  \caption{Convergence results for low-rank matrix completion with low-rank projections.
  }\label{table:empirical}
\end{center}
}
\end{table*}\renewcommand{\arraystretch}{1}

\pagebreak

\bibliographystyle{plain}
\bibliography{bib}

%\section*{Appendix}

%\begin{lemma}
%For the Matrix Completion objective $f(\X) = \frac{1}{2}\sum_{(i,j)\in{}S}(\E_{i,j}\bullet\X-b_{i,j})^2$ it holds that for any optimal solution $\X^*$ and for any feasible point $\X$ that $\Vert{\nabla{}f(\X) - \nabla{}f^*}\Vert_F \leq \sqrt{2\left({f(\X) - f^*}\right)}$.
%\end{lemma}
\end{document}